\documentclass[12pt,twoside]{amsart}
\usepackage{amssymb}
\usepackage{amsmath}
\usepackage{xypic}
\usepackage{color}

\dedicatory{Dedicated to Professor Shigeru Mukai on the occasion of his 
sixtieth birthday}

\title{Fundamental theorems for semi log canonical pairs}
\author{Osamu Fujino} 
\date{2013/10/25, version 1.92}
\address{Department of Mathematics, Faculty of Science, 
Kyoto University, Kyoto 606-8502, Japan}
\email{fujino@math.kyoto-u.ac.jp}

\newcommand{\cond}[0]{\operatorname{\mathfrak{cond}}}
\newcommand{\codim}[0]{\operatorname{codim}}

\newcommand{\ncp}[0]{{\operatorname{ncp}}}
\newcommand{\Spec}[0]{\operatorname{Spec}}
\newcommand{\Pic}[0]{\operatorname{Pic}}
\newcommand{\Exc}[0]{\operatorname{Exc}}
\newcommand{\Supp}[0]{\operatorname{Supp}}

\newcommand{\snc}[0]{{\operatorname{snc}}}
\newcommand{\sn}[0]{{\operatorname{snc2}}}

\newcommand{\Sing}[0]{\operatorname{Sing}}
\newcommand{\Bir}[0]{\operatorname{Bir}}
\newcommand{\Aut}[0]{\operatorname{Aut}}

\newtheorem{thm}{Theorem}[section]
\newtheorem{lem}[thm]{Lemma}
\newtheorem{cor}[thm]{Corollary}
\newtheorem{prop}[thm]{Proposition}
\newtheorem{conj}[thm]{Conjecture}

\theoremstyle{definition}
\newtheorem{ex}[thm]{Example}
\newtheorem{defn}[thm]{Definition}
\newtheorem{rem}[thm]{Remark}
\newtheorem*{ack}{Acknowledgments} 
\newtheorem*{notation}{Notation}
\newtheorem{say}[thm]{}
\newtheorem{step}{Step}

\begin{document}

\maketitle 

\begin{abstract}
We prove that every quasi-projective semi log canonical 
pair has a natural quasi-log structure with several good properties. 
It implies that various vanishing theorems, torsion-free theorem, 
and 
the cone and contraction theorem hold for semi log canonical pairs. 
\end{abstract}

\tableofcontents

\section{Introduction}\label{sec1}

In this paper, we give a {\em{natural}} quasi-log structure 
(cf.~\cite{ambro}) to an arbitrary quasi-projective semi log canonical pair. 
Note that a stable pointed curve is a typical example of semi log canonical pairs. 
As applications, we obtain various Kodaira type vanishing theorems, the cone and contraction 
theorem, and so on, for semi log canonical pairs. 
The notion of semi log canonical singularities was introduced in \cite{ks} in order to 
investigate deformations of surface singularities and compactifications of 
moduli spaces for surfaces of general type. 
By the recent developments of the minimal model 
program, we know that the appropriate singularities to permit on the varieties 
at the boundaries of moduli spaces are 
semi log canonical (see, for example, \cite{alexeev0}, \cite{alexeev}, \cite{kollar}, 
\cite[Part III]{hk}, \cite{kovacs}, \cite{kovacs2}, and so on). 
We note that the approach to the moduli 
problems in \cite{ks} is not directly related to Mumford's 
geometric invariant theory. However, the notion of semi log canonical singularities appears to be 
natural from the geometric invariant theoretic viewpoint by \cite{odaka}. 
Moreover, semi log canonical pairs play crucial roles in our inductive 
treatment of the log abundance conjecture (see, 
for example, \cite{fujino-abun} and \cite{fujino-gongyo}). 
Therefore, it is very important to establish some foundational 
techniques to investigate semi log canonical pairs. 
To the best knowledge of the author, there were no attempts to prove the fundamental theorems 
of the log minimal model program, for example, the cone and contraction theorem, various 
Kodaira type vanishing theorems, and so on, for {\em{semi log canonical pairs}}.  
For a different approach to semi log canonical pairs by J\'anos Koll\'ar, see 
\cite{kollar-book}, where he discusses his gluing theory for stable pairs, that is, 
semi log canonical pairs with ample log canonical divisor. 
We prove the following theorem. 

\begin{thm}\label{main}
Let $(X, \Delta)$ be a quasi-projective 
semi log canonical pair. 
Then $[X, K_X+\Delta]$ has a quasi-log structure with only qlc singularities. 
\end{thm}

Our proof of Theorem \ref{main} heavily depends on 
the recent developments of the theory of partial resolution of singularities 
for {\em{reducible}} varieties (see, for example, \cite[Section 10.4]{kollar-book}, 
\cite{bierstone-milman}, 
\cite{bierstone-p}, and so on). Precisely speaking, we prove the following theorem. 

\begin{thm}[Main theorem]\label{main2} 
Let $(X, \Delta)$ be a quasi-projective semi log canonical pair. 
Then we can construct a smooth quasi-projective variety $M$ 
with $\dim M=\dim X+1$, 
a simple normal crossing divisor $Z$ on $M$, 
a subboundary $\mathbb R$-Cartier $\mathbb R$-divisor 
$B$ on $M$, and 
a projective surjective morphism 
$h:Z\to X$ with the following properties. 
\begin{itemize}
\item[(1)] $B$ and $Z$ have no common irreducible components. 
\item[(2)] $\Supp (Z+B)$ is a simple normal crossing divisor on $M$. 
\item[(3)] $K_Z+\Delta_Z\sim _{\mathbb R}h^*(K_X+\Delta)$ such that $\Delta_Z=B|_Z$. 
\item[(4)] $h_*\mathcal O_Z(\lceil -\Delta_Z^{<1}\rceil)\simeq \mathcal O_X$. 
\end{itemize}
By the properties $(1)$, $(2)$, $(3)$, and $(4)$, $[X, K_X+\Delta]$ has 
a quasi-log structure with only qlc singularities. 
\begin{itemize}
\item[(5)] The set of slc strata of $(X, \Delta)$ gives the set of 
qlc centers of $[X, K_X+\Delta]$. This means that 
$W$ is an slc stratum of 
$(X, \Delta)$ if and only if $W$ is 
the $h$-image of some stratum of the simple normal crossing 
pair $(Z, \Delta_Z)$. 
\end{itemize}
By the property $(5)$, the above quasi-log structure of $[X, K_X+\Delta]$ is 
compatible with the original semi log canonical structure of $(X, \Delta)$. 

We note that $h_*\mathcal O_Z\simeq \mathcal O_X$ by the condition {\em{(4)}}. 
\end{thm}

\begin{rem}
In Theorem \ref{main2}, if $K_X+\Delta$ is $\mathbb Q$-Cartier, then 
we can make $B$ a $\mathbb Q$-Cartier $\mathbb Q$-divisor on $M$ satisfying 
$$K_Z+\Delta_Z\sim _{\mathbb Q}
h^*(K_X+\Delta).$$ 
It is obvious by the construction of $B$ in the proof of Theorem \ref{main2}. 
\end{rem}

By Theorem \ref{main2}, we can prove the fundamental theorems, 
that is, various Kodaira type vanishing theorems, the base point free theorem, 
the rationality 
theorem, the cone theorem, and so on, for semi log canonical pairs. 
Note that all the fundamental theorems for log canonical pairs can be proved without using the theory of 
quasi-log varieties (see 
\cite{fujino-non} and \cite{fujino-fund}). 
We also note that all the results in this section except Theorem \ref{kod-vani} are 
new even for semi log canonical {\em{surfaces}}. 

\begin{ex}Let $X$ be an equidimensional projective variety 
having only normal crossing points and pinch points. 
Then $X$ is a semi log canonical variety. 
By Theorem \ref{main2}, 
$X$ has a natural quasi-log structure. 
Therefore, all the theorems in this section hold for $X$. 
\end{ex}

Note that $h$ is not necessarily birational in Theorem \ref{main2}. 
It is a key point of the theory of quasi-log 
varieties. 

\begin{rem}[Double covering trick due to Koll\'ar]\label{rem13}
If the irreducible components of $X$ have no self-intersection in codimension one, then 
we can make $h:Z\to X$ birational in Theorem \ref{main2}. 
For some applications, by using Koll\'ar's double covering trick (see Lemma \ref{double}), 
we can reduce the problem to the case when the irreducible components of $X$ have 
no self-intersection in codimension one. 
This reduction sometimes makes the problem much easier not only technically but 
also psychologically. 
\end{rem}

Let us quickly recall a very important example. 
We recommend the reader to see \cite[Section 3.6]{fujino-what} for related topics. 

\begin{say}[Whitney umbrella] 
Let us consider the Whitney umbrella 
$X=(x^2-y^2z=0)\subset \mathbb A^3$. In this case, 
we take a blow-up $Bl _C\mathbb A^3\to \mathbb A^3$ 
of $\mathbb A^3$ along $C=(x=y=0)\subset \mathbb A^3$ and 
set $M=Bl _C\mathbb A^3$ and $Z=X'+E$, where 
$X'$ is the strict transform of $X$ on $M$ and $E$ is the exceptional divisor of the 
blow-up. Then the projective surjective morphism $h:Z\to X$ gives 
a quasi-log structure on the pair $(X, 0)$. 
Since $Z$ is a quasi-projective simple normal crossing variety, 
we can easily use the theory of mixed Hodge structures and obtain various vanishing theorems 
for $X$. It is a key point of the theory of quasi-log 
varieties. 
Note that $K_Z=h^*K_X$ and $h_*\mathcal O_Z\simeq \mathcal O_X$. 
Although $g=h|_{X'}:X'\to X$ is a resolution of singularities, 
it does not have good properties. 
This is because $X$ is not normal and $\mathcal O_X\subsetneq g_*\mathcal O_{X'}$. 
\end{say}
 
 By Theorem \ref{main2}, we can prove the following vanishing theorem 
 (see \cite[Theorem 1-2-5]{kmm}). 
 It is a generalization of the Kawamata--Viehweg vanishing theorem. 
 
\begin{thm}[Vanishing theorem I]\label{vani-thm}
Let $(X, \Delta)$ be a semi log canonical pair and let $\pi:X\to S$ be 
a projective morphism onto 
an algebraic variety $S$. 
Let $D$ be a Cartier divisor on $X$, or a 
Weil divisor on $X$ whose support 
does not contain any irreducible components of the conductor of $X$ and 
which is $\mathbb Q$-Cartier. 
Assume that 
$D-(K_X+\Delta)$ is $\pi$-ample. 
Then $R^i\pi_*\mathcal O_X(D)=0$ for 
every $i>0$. 
\end{thm}

As a special case of Theorem \ref{vani-thm}, 
we have the Kodaira vanishing theorem for semi log canonical varieties 
(cf.~\cite[Corollary 6.6]{kss}). 

\begin{thm}[Kodaira vanishing theorem]\label{kod-vani} 
Let $X$ be a projective semi log canonical variety and let $\mathcal L$ be an ample line bundle on $X$. 
Then $H^i(X, \omega_X\otimes \mathcal L)=0$ for every $i>0$. 
\end{thm}

Note that the dual form of the Kodaira vanishing theorem, that is, 
$H^i(X, \mathcal L^{-1})=0$ for $i<\dim X$, is treated by Kov\'acs--Schwede--Smith. 
For the details, see \cite[Corollary 6.6]{kss}. 
In general, $X$ is not Cohen--Macaulay. Therefore, the dual form of the Kodaira 
vanishing theorem does not always hold. 
The arguments in \cite{kss} are based on the theory of Du Bois singularities (see, for example, 
\cite{kss}, \cite{kk}, and \cite[Chapter 6]{kollar-book}). In this paper, we do not use the notion of 
Du Bois singularities. 

To the best knowledge of the author, 
even the following basic vanishing result for stable $n$-folds with $n\geq 2$ is new. 
It is a direct consequence of Theorem \ref{vani-thm}. 

\begin{cor}[Vanishing theorem for stable varieties]\label{cor-new} 
Let $X$ be a stable variety, that is, a projective 
semi log canonical variety such that $K_X$ is ample. 
Then $H^i(X, \mathcal O_X(mK_X))=0$ for every $i>0$ and $m\geq 2$. In particular, 
$$\chi (X, \mathcal O_X(mK_X))=\dim _{\mathbb C}H^0(X, \mathcal O_X(mK_X))\geq 0$$ for every 
$m\geq 2$. 
\end{cor} 
 
Theorem \ref{vani-thm} is a special case of the following theorem:~Theorem \ref{vani-thm2}. 
It is a generalization of the vanishing theorem of Reid--Fukuda type. 
The proof of Theorem \ref{vani-thm2} is much harder than that of 
Theorem \ref{vani-thm}. 

\begin{thm}[Vanishing theorem II]\label{vani-thm2}
Let $(X, \Delta)$ be a semi log canonical pair and let $\pi:X\to S$ be 
a projective morphism onto 
an algebraic variety $S$. 
Let $D$ be a Cartier divisor on $X$, or a 
Weil divisor on $X$ whose support 
does not contain any irreducible components of the conductor of $X$  and 
which is $\mathbb Q$-Cartier. 
Assume that 
$D-(K_X+\Delta)$ is nef and log big over $S$ with respect to $(X, \Delta)$. 
Then $R^i\pi_*\mathcal O_X(D)=0$ for 
every $i>0$. 
\end{thm}

For applications to the study of linear systems on semi log canonical 
pairs, Theorem \ref{vani-thm3}, 
which is a generalization of the Kawamata--Viehweg--Nadel 
vanishing theorem,  
is more convenient (see, for example, \cite[Theorem 8.1]{fujino-fund}). 
See also Remark \ref{rem42} below. 

\begin{thm}[Vanishing theorem III]\label{vani-thm3}
Let $(X, \Delta)$ be a semi log canonical pair and let $\pi:X\to S$ be 
a projective morphism onto 
an algebraic variety $S$. 
Let $D$ be a Cartier divisor on $X$ such that 
$D-(K_X+\Delta)$ is nef and log big over $S$ with 
respect to $(X, \Delta)$.  
Assume that 
$X'$ is a union of some slc strata of $(X, \Delta)$ with the reduced structure. 
Let $\mathcal I_{X'}$ be the defining 
ideal sheaf of $X'$ on $X$. 
Then $R^i\pi_*(\mathcal I_{X'}\otimes \mathcal O_X(D))=0$ for 
every $i>0$. 
\end{thm}

Note that our proof of the vanishing theorems uses the theory of the 
mixed Hodge structures on cohomology groups with compact support 
(cf.~\cite[Chapter 2]{book}). 
Therefore, Theorems \ref{vani-thm}, \ref{kod-vani}, \ref{vani-thm2}, and 
\ref{vani-thm3} are Hodge theoretic 
(see also \cite{fujino-on}, \cite{fujino-fund}, \cite{fujino-vanishing}, and \cite{fujino-inj}). 

We can also prove a generalization of Koll\'ar's torsion-free theorem for 
semi log canonical pairs (see \cite[Theorem 1-2-7]{kmm}, 
\cite[Theorem 2.2]{high}, \cite[Theorem 6.3 (iii)]{fujino-fund}, and so on). 

\begin{thm}[Torsion-free theorem]\label{thm-torsion} 
Let $(X, \Delta)$ be a semi log canonical pair 
and let $\pi:X\to S$ be a projective morphism onto an algebraic variety $S$. 
Let $D$ be a Cartier divisor on $X$, or a 
Weil divisor on $X$ whose support 
does not contain any irreducible components of the conductor of $X$ and 
which is $\mathbb Q$-Cartier. 
Assume that $D-(K_X+\Delta)$ 
is $\pi$-semi-ample. 
Then every associated prime of $R^i\pi_*\mathcal O_X(D)$ 
is the generic point of the $\pi$-image of some slc stratum of $(X, \Delta)$ for every $i$. 
\end{thm}

By the following adjunction formula, which is a direct consequence of 
Theorem \ref{main2}, we can apply the theory of quasi-log 
varieties to any union of some slc strata of a quasi-projective semi log canonical pair $(X, \Delta)$. 

\begin{thm}[Adjunction]\label{thm-adj} 
Let $(X, \Delta)$ be a quasi-projective semi log canonical pair and let $X'$ be a union 
of some slc strata of $(X, \Delta)$ with the 
reduced structure. 
Then $[X', (K_X+\Delta)|_{X'}]$ has a natural quasi-log structure 
with only qlc singularities induced by the quasi-log structure on $[X, K_X+\Delta]$ 
constructed in 
{\em{Theorem \ref{main2}}}. 
Therefore, $W$ is a qlc center of $[X', (K_X+\Delta)|_{X'}]$ if and only if 
$W$ is an slc stratum of 
$(X, \Delta)$ contained in $X'$. 
In particular, $X'$ is semi-normal. 
\end{thm}

Theorem \ref{vani-thm4}, which is a vanishing theorem for 
a union of some slc strata, is very powerful for various applications (cf.~\cite[Theorem 11.1]{fujino-fund}). 
See Remark \ref{rem15} below. 

\begin{thm}[Vanishing theorem IV]\label{vani-thm4} 
Let $(X, \Delta)$ be a semi log canonical 
pair and let $\pi:X\to S$ be a projective 
morphism onto an algebraic variety $S$. 
Assume that 
$X'$ is a union of some slc strata of $(X, \Delta)$ with the reduced structure. 
Let $L$ be a Cartier divisor on $X'$ such that 
$L-(K_X+\Delta)|_{X'}$ is nef over $S$. 
Assume that 
$(L-(K_X+\Delta)|_{X'})|_W$ is big over $S$ where 
$W$ is any slc stratum of $(X, \Delta)$ contained in $X'$. 
Then $R^i(\pi|_{X'})_*\mathcal O_{X'}(L)=0$ for every $i>0$. 
\end{thm}

Theorem \ref{vani-thm4} directly follows from Theorem \ref{thm-adj} by the 
theory of quasi-log varieties. 

By Theorem \ref{main2}, we can use the theory of quasi-log varieties 
to investigate semi log canonical pairs.
The base point free theorem holds for semi log canonical pairs 
(cf.~\cite[Theorem 3-1-1]{kmm}). 

\begin{thm}[Base point free theorem]\label{bpf} 
Let $(X, \Delta)$ be a semi log canonical pair 
and let $\pi:X\to S$ be a projective morphism onto 
an algebraic variety $S$. 
Let $D$ be a $\pi$-nef Cartier divisor on $X$. 
Assume that 
$aD-(K_X+\Delta)$ is $\pi$-ample for some real number $a>0$. 
Then $\mathcal O_X(mD)$ is $\pi$-generated for every $m\gg 0$, that is, 
there exists a positive integer $m_0$ such that 
$\mathcal O_X(mD)$ is $\pi$-generated for every $m\geq m_0$. 
\end{thm}

We can prove the base point free theorem of Reid--Fukuda type 
for semi log canonical pairs (see also \cite[Theorem 0.1]{fujino-rf}, 
\cite[Section 5]{fujino-base}, and so on). It is a slight generalization of 
Theorem \ref{bpf}. 
Note that Theorem \ref{bpf} is sufficient for the contraction theorem in Theorem \ref{cone-contraction}. 

\begin{thm}[Base point free theorem II]\label{bpf2} 
Let $(X, \Delta)$ be a semi log canonical pair 
and let $\pi:X\to S$ be a projective morphism onto 
an algebraic variety $S$. 
Let $D$ be a $\pi$-nef Cartier divisor on $X$. 
Assume that 
$aD-(K_X+\Delta)$ is nef and log big over $S$ with respect to $(X, \Delta)$ 
for some real number $a>0$. 
Then $\mathcal O_X(mD)$ is $\pi$-generated for every $m\gg 0$, that is, 
there exists a positive integer $m_0$ such that 
$\mathcal O_X(mD)$ is $\pi$-generated for every $m\geq m_0$. 
\end{thm}

From some technical viewpoints, 
we give an important remark. 

\begin{rem}\label{rem15}
We can prove Theorem \ref{bpf} without using the theory of quasi-log varieties. 
The proofs of the non-vanishing theorem and the base point free 
theorem in \cite{fujino-fund} can be adapted to our situation in Theorem \ref{bpf} once we 
adopt Theorem \ref{vani-thm4}. 
For the details, see \cite[Sections 12 and 13]{fujino-fund}. 
On the other hand, the theory of quasi-log varieties seems to be 
indispensable for the proof of Theorem \ref{bpf2}. 
Therefore, the proof of Theorem \ref{bpf2} is much harder than that of 
Theorem \ref{bpf}. 
\end{rem}

It is known that the rationality theorem holds for quasi-log varieties. 
Therefore, as a consequence of Theorem \ref{main2}, we obtain the 
rationality theorem for semi log canonical pairs (cf.~\cite[Theorem 4-1-1]{kmm}). 
Note that we can obtain 
Theorem \ref{rational} as an application of Theorem \ref{vani-thm3} and 
that the proof of Theorem \ref{rational} does not need the theory of quasi-log varieties 
(see \cite[Theorem 8.1 and the proof of Theorem 15.1]{fujino-fund}). 

\begin{thm}[Rationality theorem]\label{rational} 
Let $(X, \Delta)$ be a semi log canonical pair and let $\pi:X\to S$ 
be a projective morphism onto an algebraic variety $S$. 
Let $H$ be a $\pi$-ample Cartier divisor on $X$. 
Assume that $K_X+\Delta$ is not $\pi$-nef 
and that there is a positive integer $a$ such that 
$a(K_X+\Delta)$ is $\mathbb R$-linearly equivalent to 
a Cartier divisor. 
Let $r$ be a positive real number such that 
$H+r(K_X+\Delta)$ is $\pi$-nef but not 
$\pi$-ample. 
Then $r$ is a rational 
number, and in reduced form, 
it has denominator at most $a(\dim X+1)$.  
\end{thm}

By using Theorem \ref{bpf} and Theorem \ref{rational}, 
we obtain the cone and contraction theorem for semi log canonical 
pairs. 

\begin{thm}[Cone and contraction theorem]\label{cone-contraction}
Let $(X, \Delta)$ be a semi log canonical pair 
and let $\pi:X\to S$ be a projective 
morphism onto an algebraic variety $S$. 
Then we have the following properties. 
\begin{itemize}
\item[(1)] There are {\em{(}}countably many{\em{)}} rational 
curves $C_j\subset X$ such that 
$0<-(K_X+\Delta)\cdot C_j\leq 2\dim X$, $\pi(C_j)$ is a point, 
and $$
\overline {NE}(X/S)=\overline {NE}(X/S)_{(K_X+\Delta)\geq 0}+\sum \mathbb R_{\geq 0}[C_j]. 
$$
\item[(2)] For any $\varepsilon >0$ and any $\pi$-ample 
$\mathbb R$-divisor 
$H$, 
$$
\overline {NE}(X/S)=\overline {NE}(X/S)_{(K_X+\Delta+\varepsilon H)\geq 0}+\sum _{\text{finite}}
\mathbb R_{\geq 0}
[C_j]. 
$$
\item[(3)] 
Let $F\subset \overline {NE}(X/S)$ be a $(K_X+\Delta)$-negative 
extremal face. 
Then there is a unique morphism $\varphi_F:X\to Z$ over $S$ such that 
$(\varphi_F)_*\mathcal O_X\simeq \mathcal O_Z$, 
$Z$ is projective over $S$, and that an irreducible 
curve $C\subset X$ where  
$\pi(C)$ is a point is mapped to a point by $\varphi_F$ if and only if $[C]\in F$. 
The map $\varphi_F$ is called the {\em{contraction}} associated to 
$F$. 
\item[(4)] Let $F$ and $\varphi_F$ be as in {\em{(3)}}. 
Let $L$ be a line bundle on $X$ such that 
$L\cdot C=0$ for every curve $[C]\in F$. Then there is a line bundle $M$ on $Z$ such that 
$L\simeq \varphi_F^*M$. 
\end{itemize}
\end{thm}

Although we have established the cone and contraction theorem for 
semi log canonical pairs, a simple example (see Example \ref{ex-bad}) 
shows that we can not always run the minimal model program even for 
semi log canonical {\em{surfaces}}. However, we have some nontrivial 
applications of Theorem \ref{cone-contraction} (see Section \ref{sec-mis}). 
Moreover, Kento Fujita has recently constructed 
{\em{semi-terminal modifications}} for quasi-projective 
{\em{demi-normal}} pairs by running a variant of the minimal model 
program for {\em{semi-terminal}} pairs. 
His arguments use Theorem \ref{cone-contraction} 
and Koll\'ar's gluing theory. 
For the details, see \cite{fujita2}. 

We can prove many other powerful results by translating 
the results for quasi-log varieties (see, for example, 
Corollary \ref{cor35ne}). 
For the details of the theory of quasi-log varieties, 
see \cite{book} and \cite{fujino-qlog}. 
We recommend the reader to see \cite{fujino-fund} for various vanishing 
theorems, 
the non-vanishing theorem, the base point free theorem, the cone theorem, and so on, 
for pairs $(X, \Delta)$, where 
$X$ is a {\em{normal}} variety and $\Delta$ is an effective $\mathbb R$-divisor 
on $X$ such that $K_X+\Delta$ is $\mathbb R$-Cartier. The arguments in \cite{fujino-fund} 
are independent of the theory of quasi-log varieties and 
only use {\em{normal}} varieties for the above fundamental theorems. 
In this paper, we do not need the recent advances in the minimal model 
program mainly due to Birkar--Cascini--Hacon--M\textsuperscript{c}Kernan (cf.~\cite[Part II]{hk}). 

For the abundance conjecture for semi log canonical pairs, 
see \cite{fujino-abun}, \cite{gongyo}, \cite{fujino-gongyo}, and \cite{hacon-xu}. 
These papers are independent of the techniques discussed in this paper. 
We give some results supplementary to \cite{fujino-gongyo} in Section \ref{sec-mis}. 
In this introduction, we explain only one result on the finiteness of automorphisms. 

\begin{thm}[see Theorem \ref{67}]\label{thm118}
Let $(X, \Delta)$ be a complete  semi log canonical 
pair such that 
$K_X+\Delta$ is a big $\mathbb Q$-Carteir $\mathbb Q$-divisor. 
Then 
$$
\Bir (X, \Delta)=\{ f\, |\, f:(X, \Delta)\dashrightarrow (X, \Delta)\ \text{is $B$-birational}\}
$$
is a finite group. In particular, 
$$
\Aut (X, \Delta)=\{ f\, |\, f: X\to X \ {\text{is an isomorphism such that}} \ 
\Delta=f_*^{-1}\Delta\}
$$ 
is a finite group. 
\end{thm}

For the details, see Theorem \ref{67} below. 
Theorem \ref{thm118} seems to be an important property when we consider moduli spaces of 
{\em{stable pairs}}. 

By combining the arguments in the proof of Theorem \ref{main2} with our new semi-positivity 
theorem in \cite{fuji-fuji} (see also \cite{ffs}), 
we obtain the following semi-positivity theorem in \cite{fujino-sp}. 

\begin{thm}[{see \cite[Theorem 1.8]{fujino-sp}}]
\label{new-sp} 
Let $X$ be an equidimensional variety which satisfies Serre's $S_2$ condition and is Gorenstein
in codimension one.
Let $f:X\to C$ be a projective
surjective morphism onto a smooth projective curve $C$ such that
every irreducible component of
$X$ is dominant onto $C$.
Assume that there exists a non-empty Zariski open set $U$ of $C$ such that
$f^{-1}(U)$ has only semi log canonical singularities.
Then $f_*\omega_{X/C}$ is semi-positive. 

Assume further that $\omega_{X/C}^{[k]}:=(\omega_{X/C}^{\otimes k})^{**}$ 
is locally free and $f$-generated for some positive
integer $k$.
Then $f_*\omega_{X/C}^{[m]}$ is semi-positive for every $m\geq 1$. 
\end{thm}

Theorem \ref{new-sp} implies that the moduli functor of stable varieties is semi-positive in the 
sense of Koll\'ar (see \cite[2.4.~Definition]{kollar-proj}). Therefore, 
Theorem \ref{new-sp} plays crucial roles for the projectivity of the moduli spaces of 
stable varieties. 
For the details, see \cite{kollar-proj}, \cite{fuji-fuji}, and \cite{fujino-sp}. The reader can find some generalizations of 
Theorem \ref{new-sp} in \cite{fujino-sp}. 

Finally, in this 
paper, we are mainly interested in {\em{non-normal}} algebraic varieties. 
So we have to be careful about some basic definitions. 

\begin{say}[Big $\mathbb R$-Cartier $\mathbb R$-divisors]
Let $X$ be a 
{\em{non-normal}} complete irreducible algebraic variety and let 
$D$ be a $\mathbb Q$-Cartier $\mathbb Q$-divisor on $X$ such that 
$m_0D$ is Cartier for some positive integer $m_0$. 
We can consider the asymptotic behavior of $\dim H^0(X, \mathcal O_X(mm_0D))$ for 
$m\to \infty$ since $\mathcal O_X(mm_0D)$ is a well-defined line bundle on $X$ associated to $mm_0D$. 
Therefore, there are no difficulties to define big $\mathbb Q$-Cartier $\mathbb Q$-divisors on $X$. 
Let $B$ be  an $\mathbb R$-Cartier $\mathbb R$-divisor, that is, 
a finite $\mathbb R$-linear combination of Cartier divisors, on $X$. 
In this case, there are some difficulties to consider the 
asymptotic behavior of 
$\dim H^0(X, \mathcal O_X(mB))$ for $m\to \infty$. 
It is because the meaning of $\mathcal O_X(mB)$ is not clear. 
It may happen that the support of $mB$ is contained in the singular locus of $X$. 
Therefore, we have to discuss the definition and the basic properties of big 
$\mathbb R$-Cartier $\mathbb R$-divisors 
on {\em{non-normal}} complete irreducible varieties. 
\end{say}

We summarize the contents of this paper. 
In Section \ref{sec2}, we collect some basic definitions and 
results. 
Section \ref{sec-app-qlog} contains supplementary results for the 
theory of quasi-log varieties. Section \ref{sec3} is devoted to the proof of the main theorem:~Theorem \ref{main2}. 
The proof heavily depends on the recent developments of the theory 
of partial resolution of 
singularities 
for {\em{reducible}} varieties (cf.~\cite[Section 10.4]{kollar-book}, \cite{bierstone-milman}, \cite{bierstone-p}). 
In Section \ref{sec4}, we treat the fundamental theorems in Section \ref{sec1} 
as applications of Theorem \ref{main2}. 
In Section \ref{sec-mis}, we discuss miscellaneous applications, 
for example, the base point free theorem for $\mathbb R$-divisors, 
a generalization of Koll\'ar's effective base point free theorem for semi log canonical 
pairs, Shokurov's polytope for semi log canonical pairs, depth of sheaves on 
slc pairs, semi log canonical morphisms, 
the finiteness of $B$-birational automorphisms for stable pairs, and so on. 
In Section \ref{sec-appendix}, which is an appendix, 
we discuss the notion of big $\mathbb R$-divisors 
on {\em{non-normal}} algebraic varieties because there are no good references on this topic. 

We fix the basic notation. For the standard notation of the log minimal model 
program, see, for example, \cite{fujino-fund}.  

\begin{notation} 
Let $B_1$ and $B_2$ be two $\mathbb R$-Cartier $\mathbb R$-divisors 
on a variety $X$. Then $B_1$ is linearly (resp.~$\mathbb Q$-linearly, or $\mathbb R$-linearly) equivalent to 
$B_2$, denoted by $B_1\sim B_2$ (resp.~$B_1\sim _{\mathbb Q} B_2$, or $B_1\sim _{\mathbb R}B_2$) if 
$$
B_1=B_2+\sum _{i=1}^k r_i (f_i)
$$ 
such that $f_i \in \Gamma (X, \mathcal K_X^*)$ and $r_i\in \mathbb Z$ (resp.~$r_i \in \mathbb Q$, or $r_i \in 
\mathbb R$) for every $i$. Here, $\mathcal K_X$ is the sheaf of total quotient rings of 
$\mathcal O_X$ and $\mathcal K_X^*$ is the sheaf of invertible elements in the sheaf of rings $\mathcal K_X$. 
We note that 
$(f_i)$ is a {\em{principal Cartier divisor}} associated to $f_i$, that is, 
the image of $f_i$ by 
$
\Gamma (X, \mathcal K_X^*)\to\Gamma (X, \mathcal K_X^*/\mathcal O_X^*)$, 
where $\mathcal O_X^*$ is the sheaf of invertible elements in $\mathcal O_X$. 

Let $f:X\to Y$ be a morphism. 
If there is an $\mathbb R$-Cartier $\mathbb R$-divisor $B$ on $Y$ such that 
$$
B_1\sim _{\mathbb R}B_2+f^*B, 
$$ 
then $B_1$ is said to be {\em{relatively $\mathbb R$-linearly equivalent}} to 
$B_2$. It is denoted by $B_1\sim _{\mathbb R, f}B_2$. 

When $X$ is complete, $B_1$ is {\em{numerically equivalent}} to $B_2$, denoted by 
$B_1\equiv B_2$, if $B_1\cdot C=B_2\cdot C$ for every curve $C$ on $X$. 

Let $D$ be a $\mathbb Q$-divisor (resp.~an $\mathbb R$-divisor) 
on an equidimensional variety $X$, that is, 
$D$ is a finite formal $\mathbb Q$-linear (resp.~$\mathbb R$-linear) combination 
$$D=\sum _i d_i D_i$$ of irreducible 
reduced subschemes $D_i$ of codimension one. 
We define the {\em{round-up}} $\lceil D\rceil =\sum _i \lceil d_i \rceil D_i$ (resp.~{\em{round-down}} 
$\lfloor D\rfloor =\sum _i \lfloor d_i \rfloor D_i$), where 
every real number $x$, $\lceil x\rceil$ (resp.~$\lfloor x\rfloor$) is the integer 
defined by $x\leq \lceil x\rceil <x+1$ 
(resp.~$x-1<\lfloor x\rfloor \leq x$). The 
{\em{fractional part}} $\{D\}$ of $D$ denotes $D-\lfloor D\rfloor$. We set 
$$D^{<1}=\sum _{d_i<1}d_i D_i,  \quad \text{and}\quad D^{=1}=\sum _{d_i=1}D_i.$$ 
We call $D$ a {\em{boundary}} (resp.~{\em{subboundary}}) $\mathbb R$-divisor if 
$0\leq d_i\leq 1$ (resp.~$d_i\leq 1$) for every $i$. 

Let $X$ be a normal variety and let $\Delta$ be an 
$\mathbb R$-divisor on $X$ 
such that $K_X+\Delta$ is $\mathbb R$-Cartier. 
Let $f:Y\to X$ be 
a resolution such that $\Exc(f)\cup f^{-1}_*\Delta$, 
where $\Exc (f)$ is the exceptional locus of $f$ 
and $f^{-1}_*\Delta$ is 
the strict transform of $\Delta$ on $Y$,  
has a simple normal crossing support. We can 
write 
$$K_Y=f^*(K_X+\Delta)+\sum _i a_i E_i. 
$$
We say that $(X, \Delta)$ 
is {\em{sub log canonical}} ({\em{sub lc}}, for short) if $a_i\geq -1$ for every $i$. 
We usually write $a_i= a(E_i, X, \Delta)$
and call it the {\em{discrepancy coefficient}} of 
$E_i$ with respect to $(X, \Delta)$. 
If $(X, \Delta)$ is sub log canonical and $\Delta$ is effective, then 
$(X, \Delta)$ is called {\em{log canonical}} ({\em{lc}}, for short). 
We note that we can define $a(E_i, X, \Delta)$ in more general 
settings (see \cite[Definition 2.4]{kollar-book}). 

If $(X, \Delta)$ is sub log canonical and 
there exist a resolution $f:Y\to X$ and a divisor $E$ on $Y$ such 
that $a(E, X, \Delta)=-1$, then $f(E)$ is called a 
{\em{log canonical center}} (an {\em{lc center}}, for short) with respect to $(X, \Delta)$. 

Let $X$ be a smooth projective variety and let $D$ be an 
$\mathbb R$-Cartier $\mathbb R$-divisor on $X$. 
Then $\kappa (X, D)$ denotes {\em{Iitaka's $D$-dimension}} of 
$D$ (see, for example, \cite[Chapter II.~3.2.~Definition]{nakayama2}). 

A pair $[X, \omega]$ consists of a 
scheme $X$ and an $\mathbb R$-Carteir $\mathbb R$-divisor 
$\omega$ on $X$. 
In this paper, $X$ is always a {\em{variety}}, that is, 
$X$ is a reduced separated scheme of finite 
type over $\Spec \mathbb C$.
\end{notation}

\begin{ack}
The author was partially supported by the Grant-in-Aid for Young Scientists 
(A) $\sharp$24684002 from JSPS. 
He would like to thank Professor Noboru Nakayama for discussions on 
the topics in Section \ref{sec-appendix} and Kento Fujita for informing him of his interesting 
example. 
He would like to 
thank Professor J\'anos Koll\'ar for giving him a preliminary version of his book \cite{kollar-book}. 
He would like to thank Professor 
Edward Bierstone for sending him \cite{bierstone-p}, 
which is a key ingredient of this paper. 
Finally, he would like to thank the referee for useful comments. 
\end{ack}

We will work over $\mathbb C$, the field of complex numbers, 
throughout this paper. Note that, by the Lefschetz principle, all the results hold over 
any algebraically closed field $k$ of characteristic zero. 
In this paper, we will use the notion of quasi-log varieties introduced by Florin Ambro in 
\cite{ambro}, which has not yet been so familiar even to the experts of the log minimal model 
program. Therefore we recommend the reader to take a glance at \cite{fujino-qlog} for a gentle 
introduction to the theory of quasi-log varieties before reading this paper. 

\section{Preliminaries}\label{sec2}

In this section, we collect some basic definitions and results. 
First, let us recall the definition of {\em{conductors}}. 

\begin{defn}[Conductor]\label{def-cond} 
Let $X$ be an equidimensional variety which satisfies Serre's $S_2$ condition 
and is normal crossing in codimension one and let 
$\nu:X^\nu\to X$ be the normalization. 
Then the {\em{conductor ideal}} of $X$ is defined by 
$$
\cond_X:=\mathcal H om _{\mathcal O_X}(\nu_*\mathcal O_{X^\nu}, 
\mathcal O_X)\subset \mathcal O_X. 
$$ 
The {\em{conductor}} $\mathcal C_X$ of $X$ is the subscheme defined by $\cond_X$. 
Since $X$ satisfies Serre's $S_2$ condition and 
$X$ is normal crossing in codimension one, 
$\mathcal C_X$ is a reduced closed subscheme of pure codimension one in $X$. 
\end{defn}

\begin{defn}[Double normal crossing points and pinch points]
An $n$-dimensional singularity $(x\in X)$ is called a 
{\em{double normal crossing point}} if it is analytically (or formally) isomorphic to 
$$
\left(0\in (x_0x_1=0)\right)\subset (0\in \mathbb C^{n+1}). 
$$
It is called a {\em{pinch point}} if 
it is analytically (or formally) isomorphic to 
$$
(0\in (x_0^2=x_1x_2^2))\subset (0\in \mathbb C^{n+1}). 
$$
\end{defn}

We recall the definition of {\em{semi log canonical pairs}}. 

\begin{defn}[Semi log canonical pairs]\label{slc-pairs}
Let $X$ be an 
equidimensional algebraic variety that 
satisfies Serre's $S_2$ condition and 
is normal crossing in codimension one. 
Let $\Delta$ be an effective $\mathbb R$-divisor 
whose support does not contain any irreducible components 
of the conductor of $X$. 
The pair $(X, \Delta)$ is called a {\em{semi log canonical pair}} (an {\em{slc pair}}, 
for short) 
if 
\begin{itemize}
\item[(1)] $K_X+\Delta$ is $\mathbb R$-Cartier, and 
\item[(2)] $(X^\nu, \Theta)$ is log canonical, 
where $\nu:X^\nu\to X$ is the normalization and $K_{X^\nu}+\Theta=
\nu^*(K_X+\Delta)$. 
\end{itemize}
\end{defn}

We introduce the notion of {\em{semi log canonical centers}}. 
It is a direct generalization of the notion of log canonical centers for 
log canonical pairs. 

\begin{defn}[Slc center]\label{defn-slccenter}
Let $(X, \Delta)$ be a semi log canonical pair and let $\nu:X^\nu\to X$ be 
the normalization. 
We set  
$$
K_{X^\nu}+\Theta=\nu^*(K_X+\Delta). 
$$
A closed subvariety $W$ of $X$ is called a {\em{semi log canonical center}} 
(an {\em{slc center}}, for short) {\em{with 
respect to $(X, \Delta)$}} if there exist a resolution of singularities $f: Y\to X^\nu$ and 
a prime divisor $E$ on $Y$ such that 
the discrepancy coefficient $a(E, X^\nu, \Theta)=-1$ and $\nu\circ f(E)=W$. 
\end{defn}

For our purposes, it is very convenient to introduce the notion of 
{\em{slc strata}} for semi log canonical pairs. 

\begin{defn}[Slc stratum]\label{stra} 
Let $(X, \Delta)$ be a semi log canonical pair. 
A subvariety $W$ of $X$ is called an {\em{slc stratum}} of 
the pair $(X, \Delta)$ if 
$W$ is a semi log canonical center with respect to $(X, \Delta)$ or $W$ is an 
irreducible component of $X$. 
\end{defn}

In this paper, we mainly discuss {\em{non-normal}} 
algebraic varieties and divisors on them. 
We have to be careful when we use Weil divisors on non-normal varieties. 

\begin{say}[Divisorial sheaves]\label{divisorial} 
Let $D$ be a Weil divisor on a semi log canonical pair $(X, \Delta)$ 
whose support does not contain any irreducible components of the conductor of $X$. 
Then the reflexive sheaf $\mathcal O_X(D)$ is well-defined. 
In this paper, we do not discuss Weil divisors whose supports 
contain some irreducible components of the conductor of $X$. 
Note that if $D$ is a Cartier divisor on $X$ then $\mathcal O_X(D)$ is 
a well-defined invertible sheaf on $X$ without any assumptions on the support of $D$. 
\end{say}

For the details, we recommend the reader to see \cite[5.6]{kollar-book} and 
\cite[Chapter 16]{fa} by Alesio Corti. 
See also \cite[Sections 1 and 2]{harts}. 
The remarks in \ref{divisorial} are sufficient for our purposes in this paper. So we 
do not pursue the definition of $\mathcal O_X(D)$ any more.  

Next, let us recall the definition of {\em{nef and log big 
$\mathbb R$-Cartier $\mathbb R$-divisors}} on semi log canonical 
pairs. For the details of big $\mathbb R$-Cartier $\mathbb R$-divisors, 
see Section \ref{sec-appendix}. 

\begin{defn}[Nef and log big divisors on slc pairs]\label{def27}
Let $(X, \Delta)$ be a semi log canonical pair and let $\pi:X\to S$ be a proper surjective 
morphism onto an algebraic variety $S$. 
Let $D$ be a $\pi$-nef $\mathbb R$-Cartier $\mathbb R$-divisor on $X$. 
Then $D$ is {\em{nef and log big over $S$ with respect to $(X, \Delta)$}} 
if 
$D|_W$ is big 
over $S$ for every slc stratum of $(X, \Delta)$. 
\end{defn}

Finally, let us recall the definition of {\em{simple normal crossing pairs}}. 
In \cite{kollar-book} and \cite{bierstone-p}, a simple normal crossing pair is called 
a {\em{semi-snc pair}}. 

\begin{defn}[Simple normal crossing pairs]\label{def-snc} 
We say that the pair $(X, D)$ is {\em{simple normal crossing}} at 
a point $a\in X$ if $X$ has a Zariski open neighborhood $U$ of $a$ that can be embedded in a smooth 
variety 
$Y$, 
where $Y$ has regular system of parameters $(x_1, \cdots, x_p, y_1, \cdots, y_r)$ at 
$a=0$ in which $U$ is defined by a monomial equation 
$$
x_1\cdots x_p=0
$$ 
and $$
D=\sum _{i=1}^r \alpha_i(y_i=0)|_U, \quad  \alpha_i\in \mathbb R. 
$$ 
We say that $(X, D)$ is a {\em{simple normal crossing pair}} if it is simple normal crossing at every point of $X$. 
We say that a simple normal crossing pair $(X, D)$ is {\em{embedded}} 
if there exists a closed embedding $\iota:X\to M$, where 
$M$ is a smooth variety of $\dim X+1$. 
If $(X, 0)$ is a simple normal crossing pair, then $X$ is called a {\em{simple normal crossing 
variety}}. If $X$ is a simple normal crossing variety, then $X$ has only Gorenstein singularities. 
Thus, it has an invertible dualizing sheaf $\omega_X$. 
Therefore, we can define the {\em{canonical divisor $K_X$}} such that 
$\omega_X\simeq \mathcal O_X(K_X)$. 
It is a Cartier divisor on $X$ and is well-defined up to linear equivalence. 

Let $X$ be a simple normal crossing variety and let $X=\bigcup _{i\in I}X_i$ be the 
irreducible decomposition of $X$. 
A {\em{stratum}} of $X$ is an irreducible component of $X_{i_1}\cap \cdots \cap X_{i_k}$ for some 
$\{i_1, \cdots, i_k\}\subset I$. 

Let $X$ be a simple normal crossing variety and 
let $D$ be a Cartier divisor on $X$. 
If $(X, D)$ is a simple normal crossing pair and $D$ is reduced, 
then $D$ is called a {\em{simple normal crossing divisor}} on $X$. 

Let $(X, D)$ be a simple normal crossing pair such that 
$D$ is a subboundary $\mathbb R$-divisor 
on $X$. 
Let $\nu:X^\nu \to X$ be the normalization. 
We define $\Xi$ by the formula 
$$
K_{X^\nu}+\Xi=\nu^*(K_X+D). 
$$ 
Then a {\em{stratum}} of $(X, D)$ is an irreducible component of $X$ or the $\nu$-image 
of a log canonical center of $(X^\nu, \Xi)$. 
We note that $(X^\nu, \Xi)$ is sub log canonical.  
When $D=0$, 
this definition is compatible with the above definition of the strata of $X$. 
When $D$ is a boundary $\mathbb R$-divisor, 
$W$ is a stratum of $(X, D)$ if and only if 
$W$ is an slc stratum of $(X, D)$. Note that 
$(X, D)$ is semi log canonical if $D$ is a boundary $\mathbb R$-divisor. 
\end{defn}

The author learned the following interesting example from Kento Fujita (cf.~\cite[Remark 1.9]{kollar-book}). 

\begin{ex}\label{rem-fujita} 
Let $X_1=\mathbb P^2$ and let $C_1$ be a line on $X_1$. 
Let $X_2=\mathbb P^2$ and let 
$C_2$ be a smooth conic on $X_2$. 
We fix an isomorphism $\tau:C_1\to C_2$. By gluing 
$X_1$ and $X_2$ along $\tau:C_1\to C_2$, we obtain a simple normal crossing surface 
$X$ such that $-K_X$ is ample (cf.~\cite{fujita}). 
We can check that $X$ can not be embedded into any smooth varieties as a simple normal crossing 
divisor. 
\end{ex}

The reader can find various vanishing theorems and 
a generalization of the Fujita--Kawamata semi-positivity theorem for 
simple normal crossing pairs in \cite{book}, \cite{fujino-vanishing}, \cite{fujino-inj}, 
and \cite{fuji-fuji} 
(see also \cite{ffs}). 
All of them depend on the theory of the mixed Hodge structures on cohomology groups 
with compact support.  

\section{Supplements to the theory of quasi-log varieties}\label{sec-app-qlog}
In this section, let us give supplementary arguments to the theory of quasi-log varieties (cf.~\cite{ambro}). 
For the details of the theory of quasi-log varieties, see \cite[Chapter 3]{book} and \cite{fujino-qlog}. 

Let us introduce the notion of {\em{globally embedded simple 
normal crossing pairs}}, which is much easier than 
the notion of embedded simple normal crossing pairs from some technical 
viewpoints. 
It is obvious that 
a globally embedded simple normal crossing pair is an 
embedded simple normal crossing pair. 

\begin{defn}[Globally embedded simple normal crossing 
pairs]\label{gsnc0} 
Let $Y$ be a simple normal crossing divisor 
on a smooth 
variety $M$ and let $B$ be an $\mathbb R$-divisor 
on $M$ such that 
$\Supp (B+Y)$ is a simple normal crossing divisor and that 
$B$ and $Y$ have no common irreducible components. 
We set $\Delta_Y=B|_Y$ and consider the pair $(Y, \Delta_Y)$. 
We call $(Y, \Delta_Y)$ a {\em{globally embedded simple normal 
crossing pair}}. 
\end{defn}

Let us recall the definition of {\em{quasi-log varieties with only qlc singularities}}. 

\begin{defn}[Quasi-log varieties with only qlc singularities] \label{def51}
A {\em{quasi-log variety with only 
qlc singularities}} is a (not necessarily equidimensional) 
variety $X$ with an $\mathbb R$-Cartier $\mathbb R$-divisor 
$\omega$, and a finite collection $\{C\}$ of reduced and irreducible subvarieties of $X$ such that 
there is a proper morphism $f:(Y, \Delta_Y)\to X$ from a globally embedded simple normal crossing 
pair satisfying the following properties. 
\begin{itemize}
\item[(1)] $f^*\omega\sim _{\mathbb R}K_Y+\Delta_Y$ such that 
$\Delta_Y$ is a subboundary $\mathbb R$-divisor. 
\item[(2)] There is an isomorphism 
$$
\mathcal O_X\simeq f_*\mathcal O_Y(\lceil -\Delta_Y^{<1}\rceil). 
$$ 
\item[(3)] The collection of subvarieties $\{C\}$ coincides with 
the image of the $(Y, \Delta_Y)$-strata.  
\end{itemize}
We simply write $[X, \omega]$ to denote the above data 
$$
\bigl(X, \omega, f:(Y, \Delta_Y)\to X\bigr)
$$ 
if there is no risk of confusion.  
The subvarieties $C$ are called the {\em{qlc centers}} of $[X, \omega]$, and 
$f:(Y, \Delta_Y)\to X$ is called a {\em{quasi-log resolution}} of 
$[X, \omega]$. 
We sometimes simply say that 
$[X, \omega]$ is a {\em{qlc pair}}, or 
the pair $[X, \omega]$ is {\em{qlc}}. 
We call $\omega$ the {\em{quasi-log canonical class}} of 
$[X, \omega]$. Note that 
$\omega$ is defined up to $\mathbb R$-linear 
equivalence. 
\end{defn}

The notion of {\em{crepant log structures}} 
introduced by Koll\'ar--Kov\'acs, which is a very special but 
important case of quasi-log structures, is also useful for various applications (see, 
for example, \cite[4.4 Crepant log structures]{kollar-book}). 
For a prototype of quasi-log structures and crepant log structures, 
see \cite[Theorem 4.1]{fujino-app}. 

Let us recall the following very useful lemma. 
By this lemma, it is sufficient to treat globally embedded simple normal 
crossing pairs for the theory of qlc pairs. 

\begin{lem}[{cf.~\cite[Proposition 3.57]{book}}]\label{useful-lem}
Let $(Y,\Delta_Y)$ be an embedded simple normal crossing pair such that 
$\Delta_Y$ is a subboundary $\mathbb R$-Cartier $\mathbb R$-divisor on $Y$. 
Let $M$ be the ambient space of $Y$. 
Then we can construct a sequence of blow-ups 
$$
M_k\overset{\sigma_k}\longrightarrow M_{k-1}\overset{\sigma_{k-1}}
\longrightarrow \cdots \overset{\sigma_0}\longrightarrow M_0=M
$$
with the following properties. 
\begin{itemize}
\item[(1)] $\sigma_{i+1}:M_{i+1}\to M_i$ is the blow-up along a smooth irreducible component of 
$\Supp \Delta_{Y_i}$ for every $i$. 
\item[(2)] We set $Y_0=Y$ and $\Delta_{Y_0}=\Delta_Y$. 
We define $Y_{i+1}$ as the strict transform of $Y_i$ on $M_{i+1}$ for every $i$. 
Note that $Y_i$ is a simple normal crossing divisor on $M_i$ for every $i$. 
\item[(3)] We define $\Delta_{Y_{i+1}}$ by 
$$
K_{Y_{i+1}}+\Delta_{Y_{i+1}}=\sigma_{i+1}^*(K_{Y_i}+\Delta_{Y_i})
$$ 
for every $i$. 
\item[(4)] There exists an $\mathbb R$-divisor $B$ on $M_k$ such that 
$\Supp (B+Y_k)$ is a simple normal crossing divisor on $M_k$, 
$B$ and $Y_k$ have no common irreducible components, 
and $B|_{Y_k}=\Delta_{Y_k}$. 
\item[(5)] $\sigma_*\mathcal O_{Y_k}(\lceil -\Delta_{Y_k}^{<1}\rceil)\simeq \mathcal O_Y
(\lceil -\Delta_Y^{<1}\rceil)$ 
where $\sigma=\sigma_1\circ \cdots \circ \sigma_k:M_k\to M$. 
\end{itemize}
\end{lem}
\begin{proof}
All we have to do is to check the property (5). 
The other properties are obvious by the construction of blow-ups. 
By 
$$
K_{Y_{i+1}}+\Delta_{Y_{i+1}}=\sigma_{i+1}^*(K_{Y_i}+\Delta_{Y_i}), 
$$ 
we have 
\begin{align*}
K_{Y_{i+1}}=&\sigma_{i+1}^*(K_{Y_i}+\{\Delta_{Y_i}\}+\Delta_{Y_i}^{=1})
\\ &+\sigma_{i+1}^*\lfloor \Delta_{Y_{i}}^{<1}\rfloor-\lfloor 
\Delta_{Y_{i+1}}^{<1}\rfloor 
-\Delta_{Y_{i+1}}^{=1}-\{\Delta_{Y_{i+1}}\}. 
\end{align*}
We can easily check that 
$\sigma_{i+1}^*\lfloor \Delta_{Y_i}^{<1}\rfloor -\lfloor 
\Delta_{Y_{i+1}}^{<1}\rfloor$ is an effective 
$\sigma_{i+1}$-exceptional 
Cartier divisor on $Y_{i+1}$. 
This is because 
$a(\nu, Y_i, \{\Delta_{Y_i}\}+\Delta_{Y_i}^{=1})=-1$ for a prime divisor 
$\nu$ over $Y_{i}$ implies 
$a(\nu, Y_i, \Delta_{Y_i})=-1$ (cf.~\cite[Definition 2.4]{kollar-book}). 
Thus, we can write 
$$
\sigma_{i+1}^*\lceil -\Delta_{Y_i}^{<1}\rceil +E=\lceil 
-\Delta_{Y_{i+1}}^{<1}\rceil
$$ 
where $E$ is an effective $\sigma_{i+1}$-exceptional 
Cartier divisor on $Y_{i+1}$. 
This implies that 
$\sigma_{i+1*}\mathcal O_{Y_{i+1}}(\lceil 
-\Delta_{Y_{i+1}}^{<1}\rceil)\simeq 
\mathcal O_{Y_i}(\lceil -\Delta_{Y_i}^{<1}\rceil)$ for every $i$. 
Thus, $\sigma _*\mathcal O_{Y_k}(\lceil 
-\Delta_{Y_k}^{<1}\rceil)\simeq \mathcal O_Y(\lceil -\Delta_Y^{<1}\rceil)$. 
\end{proof}

Although we do not need the following theorem explicitly in this paper, 
it is very important and useful. 
It helps the reader to understand quasi-log structures. 

\begin{thm}[{cf.~\cite[Proposition 4.8]{ambro}, \cite[Theorem 3.45]{book}}]\label{center}
Let $[X, \omega]$ be a qlc pair. 
Then we have the following properties. 
\begin{itemize}
\item[(i)] The intersection of two qlc centers is a union of qlc centers. 
\item[(ii)] For any point $P\in X$, the set of all 
qlc centers passing through 
$P$ has a unique minimal element $W$. Moreover, $W$ is normal 
at $P$. 
\end{itemize}
\end{thm}

By Theorem \ref{main2} (5) and Theorem \ref{center}, 
we have an obvious corollary. 

\begin{cor}\label{cor35ne} 
Let $(X, \Delta)$ be a quasi-projective 
semi log canonical pair and let $W$ be a minimal slc stratum of the pair 
$(X, \Delta)$. 
Then $W$ is normal. 
\end{cor}

The following result is a key lemma for the proof of Theorem \ref{center} (ii). 
We contain it for the reader's convenience. 

\begin{lem}\label{lem-key} 
Let $f:X\to Y$ be a proper surjective 
morphism 
from a simple normal crossing variety $X$ to an irreducible 
variety $Y$. 
Assume that every stratum of $X$ is dominant onto $Y$ and that 
$f_*\mathcal O_X\simeq \mathcal O_Y$. Then 
$Y$ is normal. 
\end{lem}
\begin{proof}
Let $\nu:Y^{\nu}\to Y$ be the normalization. 
By applying \cite[Theorem 1.5]{bierstone-milman} to the graph of the rational map 
$\nu^{-1}\circ f:X\dashrightarrow Y^{\nu}$, we obtain the following 
commutative diagram: 
$$
\xymatrix{Z\ar[d]_\beta \ar[r]^\alpha &X\ar[d]^f\\ 
Y^{\nu}\ar[r]_{\nu}&Y
}
$$ 
such that 
\begin{itemize}
\item[(i)] $Z$ is a simple normal crossing variety, and 
\item[(ii)] 
there is a Zariski open set $U$ (resp.~$V$) of $Z$ (resp.~$X$) such that 
$U$ (resp.~$V$) contains the generic point of any stratum of $Z$ (resp.~$X$) 
and that $\alpha|_U:U\to V$ is an isomorphism. 
\end{itemize} 
Then it is easy to see that $\alpha_*\mathcal O_Z\simeq \mathcal O_X$. 
Therefore, 
$$
\mathcal O_Y\simeq f_*\mathcal O_X\simeq f_*\alpha_*\mathcal O_Z\simeq 
\nu_*\beta_*\mathcal O_Z\supset \nu_*\mathcal O_{Y^\nu}. 
$$ 
This implies that $\mathcal O_Y\simeq \nu_*\mathcal O_{Y^\nu}$. 
So, we obtain that $Y$ is normal. 
\end{proof}

We recommend the reader to see \cite{fujino-qlog} for the basic properties of qlc pairs. 
Note that adjunction and vanishing theorem (see, for example, \cite[Theorem 3.6]{fujino-qlog}) 
for qlc pairs is one of the most important properties of qlc pairs. 

\section{Proof of the main theorem}\label{sec3}

Let us start the proof of the main theorem:~Theorem \ref{main2}. 

\begin{proof}[Proof of {\em{Theorem \ref{main2}}}]
We divide the proof into several steps. We repeatedly use \cite{bierstone-milman}, \cite{bierstone-p}, and 
\cite[10.4.~Semi-log-resolution]{kollar-book}. 
We prove Theorem \ref{main2} simultaneously with 
Remark \ref{rem13}. 

\begin{step}\label{step1}
Let $X^{\ncp}$ denote the open subset 
of $X$ consisting of smooth points, double normal crossing points 
and pinch points. 
Then, by \cite[Theorem 1.18]{bierstone-milman}, 
there exists a morphism $f_1:X_1\to X$ which is a finite composite of admissible 
blow-ups, such that 
\begin{itemize}
\item[(i)] $X_1=X_1^{\ncp}$, 
\item[(ii)] $f_1$ is an isomorphism over $X^{\ncp}$, and 
\item[(iii)] $\Sing X_1$ maps birationally onto the closure of $\Sing X^{\ncp}$. 
\end{itemize}  
Since $X$ satisfies Serre's $S_2$ condition and $\codim _{X} (X\setminus X^\ncp)\geq 2$, 
we can easily check that $f_{1*}\mathcal O_{X_1}\simeq \mathcal O_X$. 
\end{step}

\begin{rem}[{see \cite[Corollary 10.55]{kollar-book}}]\label{rem41}
In Step \ref{step1}, we assume that the irreducible components of $X$ have no self-intersection in codimension 
one. Let $X^\sn$ be the open subset of $X$ which has only smooth points and simple normal crossing 
points of multiplicity $\leq 2$. 
Then there is a projective birational morphism $f_1:X_1\to X$ such that 
\begin{itemize}
\item[(i)] $X_1=X_1^\sn$, 
\item[(ii)] $f_1$ is an isomorphism over $X^\sn$, and 
\item[(iii)] $\Sing X_1$ maps birationally onto the closure of $\Sing X^\sn$. 
\end{itemize}
\end{rem}

\begin{step}[{cf.~\cite[Proposition 10.59]{kollar-book}}]\label{step2}
By the construction in Step \ref{step1}, 
$X_1$ is quasi-projective. 
Therefore, we can embed $X_1$ into $\mathbb P^N$. 
We pick a finite set $W\subset X_1$ such that 
each irreducible component of $\Sing X_1$ contains a point of $W$. 
We take a sufficiently large positive integer $d$ such that 
the scheme theoretic 
base locus of $|\mathcal O_{\mathbb P^N}(d)\otimes 
\mathcal I_{{\overline X}_1}|$ is $X_1$ near every point of $W$, 
where $\overline X _1$ is the closure of $X_1$ in $\mathbb P^N$ and $\mathcal I_{\overline X _1}$ is the defining 
ideal sheaf of $\overline X_1$ in $\mathbb P^N$. 
By taking a complete intersection of 
$(N-\dim X_1-1)$ general members of $|\mathcal O_{\mathbb P^N}(d)
\otimes \mathcal I_{\overline X _1}|$, we obtain $Y\supset X_1$ such that $Y$ is smooth at every point of $W$. 
Note that we used the fact that $X_1$ has only hypersurface singularities near $W$.
By replacing $Y$ with $Y\setminus (\overline X_1\setminus X_1)$, 
we may assume that $X_1$ is closed in $Y$.  
\end{step}
\begin{step}\label{step3}
Let $g:Y_2\to Y$ be a resolution, which is a finite composite 
of admissible blow-ups. 
Let $X_2$ be the strict transform of $X_1$ on $Y_2$. 
Note that $f_2=g|_{X_2}:X_2\to X_1$ is an isomorphism over 
the generic point of any irreducible component of $\Sing X_1$ because $Y$ is smooth at 
every point of $W$. 
\end{step}
\begin{step}\label{step4}
Apply \cite[Theorem 1.18]{bierstone-milman} to $X_2\subset Y_2$ 
(see also Proof of Theorem 1.18 in \cite{bierstone-milman}). 
We obtain a projective birational morphism 
$g_3:Y_3\to Y_2$, which is a finite 
composite of admissible blow-ups, 
from a smooth variety $Y_3$ 
with the following properties (i), (ii), and (iii). 
Note that $X_3$ is the strict transform of $X_2$ on $Y_3$ and 
$f_3=g_3|_{Y_3}:X_3\to X_2$. 
\begin{itemize}
\item[(i)] $X_3=X_3^\ncp$, 
\item[(ii)] $f_3$ is an isomorphism over $X_2^\ncp$, and 
\item[(iii)] $\Sing X_3$ maps birationally onto the closure of $\Sing X_2^\ncp$. 
\end{itemize} 
Let $E$ be an irreducible component of $\Sing X_3$. 
If $E\to (f_2\circ f_3)(E)$ is not birational, then we take a blow-up 
of $Y_3$ along $E$ and replace 
$X_3$ with its strict transform. 
After finitely many blow-ups, we may assume that $X_3$ satisfies (i) and 
\begin{itemize}
\item[(iv)] $\Sing X_3$ maps birationally onto $\Sing X_1$ by $f_2\circ f_3$. 
\end{itemize}
From now on, we do not require the properties (ii) and (iii) above. 
By the above constructions, we can easily check that 
$(f_2\circ f_3)_*\mathcal O_{X_3}\simeq\mathcal O_{X_1}$ since 
$X_1$ satisfies Serre's $S_2$ condition. 
\end{step}
\begin{rem}\label{rem4242}
When $X_1$ is a simple normal crossing variety, 
we apply Szab\'o's resolution lemma to the pair 
$(Y_2, X_2)$ in Step \ref{step4}. 
Then we have the following properties. 
\begin{itemize}
\item[(i)] $X_3=X_3^\snc$, and 
\item[(ii)] $f_3$ is an isomorphism over $X_2^\snc$. 
\end{itemize}
By taking more blow-ups if necessary, we may assume (i) and 
\begin{itemize}
\item[(iv)] $\Sing X_3$ maps birationally onto $\Sing X_1$ by $f_2\circ f_3$. 
\end{itemize}
\end{rem}
\begin{step}\label{step5}
We set 
$$
K_{X_1}+\Delta_1=f_1^*(K_X+\Delta) 
$$ 
and 
$$
K_{X_3}+\Delta_3=(f_1\circ f_2\circ f_3)^*(K_X+\Delta). 
$$ 
Note that $X_1$ and $X_3$ have only Gorenstein singularities. 
Therefore, 
$\Delta_1$ and $\Delta_3$ are $\mathbb R$-Cartier $\mathbb R$-divisors. 
We also note that 
the support of $\Delta_1$ (resp.~$\Delta_3$) 
does not contain any irreducible components of the conductor of $X_1$ (resp.~$X_3$). 
Let $\nu_3:X_3^\nu\to X_3$ be the normalization. 
We set  
$$
K_{X_3^\nu}+\Theta_3=\nu_3^*(K_{X_3}+\Delta_3). 
$$ 
Then the pair 
$(X_3^\nu, \Theta_3)$ is sub log canonical because 
$(X, \Delta)$ is semi log canonical. 
\end{step}
\begin{step}\label{step6}
Let $X_3^\snc$ denote the simple normal crossing locus of $X_3$. 
Let $C$ be an irreducible component 
of $X_3\setminus X_3^\snc$. 
Then $C$ is smooth and $\dim C=\dim X_3-1$. Let 
$\alpha:W\to Y_3$ be the blow-up along $C$ and let $V$ be $\alpha^{-1}(X_3)$ with 
the reduced structure. 
Then we can directly check that 
$\beta_*\mathcal O_V\simeq \mathcal O_{X_3}$ where 
$\beta=\alpha|_V$. 
We set 
$$
K_V+\Delta_V=\beta^*(K_{X_3}+\Delta_3). 
$$ 
Note that $K_V=\beta^*K_{X_3}$ and $\Delta_V=\beta^*\Delta_3$. 
Let $\nu:V^\nu\to V$ be the normalization of $V$. 
Then $(V^\nu, \Theta_{V^\nu})$ is sub log canonical, 
where $K_{V^\nu}+\Theta_{V^\nu}=\nu^*(K_V+\Delta_V)$. 
When $C$ is a double normal crossing points locus, 
it is almost obvious. 
If $C$ is a pinch points locus, 
then it follows from Lemma \ref{prop-ne} below. 
By repeating this process finitely many 
times, we obtain a projective birational 
morphism 
$g_4:Y_4\to Y_3$ from a smooth variety $Y_4$ and a simple normal crossing divisor $X_4$ on $Y_4$ with the 
following properties. 
\begin{itemize}
\item[(i)] $f_{4*}\mathcal O_{X_4}\simeq \mathcal O_{X_3}$ where 
$f_4=g_4|_{X_4}$. 
\item[(ii)] We set 
$$
K_{X_4}+\Delta_4=f_4^*(K_{X_3}+\Delta_3). 
$$
Then $(X_4^\nu, \Theta_4)$ is sub log canonical 
where 
$\nu_4:X_4^\nu\to X_4$ is the normalization and 
$K_{X_4^\nu}+\Theta_4=\nu_4^*(K_{X_4}+\Delta_4)$. 
\end{itemize}
\end{step}
\begin{rem}
We can skip Step \ref{step6} if $X_3=X_3^\snc$. 
Therefore, we can make $h:Z\to X$ birational when 
the irreducible components of $X$ have no 
self-intersection in codimension one (see Remarks \ref{rem41} and \ref{rem4242}). 
This is because $f_5$ in Step \ref{step7} below is always birational. 
\end{rem}
\begin{step}[{cf.~\cite[Section 4]{bierstone-p}}]\label{step7}
Let $U$ be the largest Zariski open subset of $X_4$ such that 
$(U, \Delta_4|_U)$ is a simple normal crossing pair. 
Then there is a projective birational morphism $g_5:Y_5\to Y_4$ given by a composite of blow-ups 
with smooth centers with the following properties. 
\begin{itemize}
\item[(i)] Let $X_5$ be the strict transform of $X_4$ on $Y_5$. 
Then $f_5=g_5|_{X_5}:X_5\to X_4$ is an isomorphism 
over $U$. 
\item[(ii)] $(X_5, f_{5*}^{-1}\Delta_4+\Exc (f_5))$ is a simple normal crossing pair, where 
$\Exc(f_5)$ is the exceptional locus of $f_5$. 
By the construction, we can check that 
$f_{5*}\mathcal O_{X_5}\simeq \mathcal O_{X_4}$. 
\end{itemize}
\end{step}
\begin{step}
We set $M=Y_5$, $Z=X_5$, and $h=f_1\circ f_2\circ 
f_3\circ f_4\circ f_5:Z=X_5\to X$. 
Note that $M$ is a smooth 
quasi-projective variety and $Z$ is a simple normal crossing 
divisor on $M$. 
We set 
$$
K_Z+\Delta_Z=h^*(K_X+\Delta). 
$$ 
Then $(Z, \Delta_Z)$ is a simple normal crossing pair by  the above construction. 
Note that 
$\Delta_Z$ is a subboundary $\mathbb R$-divisor on $Z$. 
\end{step}
For the proof of Theorem \ref{main2}, 
we have to see that 
$h_*\mathcal O_Z(\lceil -\Delta_Z^{<1}\rceil)\simeq \mathcal O_X$. 
We will prove it in the subsequent steps. 
\begin{step}
It is obvious that 
$$f_{1*}\mathcal O_{X_1}(\lceil -\Delta_1^{<1}\rceil)\simeq \mathcal O_X.$$ 
This is because $\lceil -\Delta_1^{<1}\rceil$ is effective and 
$f_1$-exceptional. 
Note that $f_{1*}\mathcal O_{X_1}\simeq \mathcal O_X$. 
\end{step}
\begin{step}
We can easily check that 
$$
\mathcal O_{X_1}\subset (f_2\circ f_3)_*\mathcal O_{X_3}
(\lceil -\Delta_3^{<1}\rceil)\subset \mathcal O_{X_1}(\lceil 
-\Delta_1^{<1}\rceil). 
$$ 
We note that $\lceil -\Delta_3^{<1}\rceil$ is 
effective. 
Therefore, $$(f_1\circ f_2\circ f_3)_*\mathcal O_{X_3}(\lceil 
-\Delta_3^{<1}\rceil)\simeq \mathcal O_X.$$ 
\end{step}
\begin{step}
We use the notation in Step \ref{step6}. 
Let $\alpha:W\to Y_3$ be the blow-up in Step \ref{step6}. 
Note that 
$\Delta_V=\beta^*\Delta_3$ and 
$K_V=\beta^*K_{X_3}$. 
Therefore, we have 
$$
0\leq \lceil -\Delta_V^{<1}\rceil \leq \beta^*(\lceil -\Delta_3^{<1}\rceil). 
$$ 
See the description of the blow-up in Lemma \ref{prop-ne} when $\alpha:W\to Y_3$ 
is a blow-up along a pinch points locus. 
Thus 
$$
\mathcal O_{X_3}\subset \beta_*\mathcal O_V(\lceil -\Delta_V^{<1}\rceil)\subset 
\mathcal O_{X_3}(\lceil 
-\Delta_3^{<1}\rceil)
$$ 
since $\beta_*\mathcal O_V\simeq \mathcal O_{X_3}$. 
Therefore, we obtain 
that 
$$
\mathcal O_{X_3}\subset f_{4*}\mathcal O_{X_4}(\lceil -\Delta_4^{<1}\rceil)\subset 
\mathcal O_{X_3}(\lceil 
-\Delta_3^{<1}\rceil).
$$ 
This implies that $(f_1\circ f_2\circ f_3\circ f_4)_*\mathcal O_{X_4}(\lceil 
-\Delta_4^{<1}\rceil)\simeq \mathcal O_X$. 
\end{step}
\begin{step}
It is easy to see that  
$$
\mathcal O_{X_4}\subset f_{5*}\mathcal O_{X_5}(\lceil 
-\Delta_5^{<1}\rceil)\subset 
\mathcal O_{X_4}(\lceil -\Delta_4^{<1}\rceil)
$$ because $f_5$ is a birational map. 
Thus 
$$
(f_1\circ f_2\circ f_3\circ f_4\circ f_5)_*\mathcal O_{X_5} 
(\lceil -\Delta_5^{<1}\rceil)\simeq \mathcal O_X. 
$$ So we obtain $f_*\mathcal O_Z(\lceil -\Delta_Z^{<1}\rceil)\simeq \mathcal O_X$. 
\end{step}
\begin{step}
By the construction, it is easy to see that 
$K_Z+\Delta_Z\sim _{\mathbb R}h^*(K_X+\Delta)$ and 
that $W$ is an slc stratum of $(X, \Delta)$ if and only if $W$ is the 
$h$-image of some stratum of the simple normal crossing 
pair $(Z, \Delta_Z)$ (cf.~Lemma \ref{prop-ne}). 
\end{step}
\begin{step}
By applying Lemma \ref{useful-lem}, we may assume that 
there is a subboundary 
$\mathbb R$-Cartier $\mathbb R$-divisor $B$ on $M$ such that 
$B$ and $Z$ have no common irreducible components, 
$\Supp (B+Z)$ is a simple normal crossing divisor on $M$, and 
$B|_Z=\Delta_Z$ after taking some blow-ups. 
\end{step}
Therefore, $h:(Z, \Delta_Z)\to X$ gives 
the pair $[X, K_X+\Delta]$ 
a quasi-log structure with 
the desired properties (1), (2), (3), (4), and (5). 
\end{proof}

The following easy local calculation played a crucial role in the proof of Theorem \ref{main2}. 

\begin{lem}\label{prop-ne}
We consider 
$$
V=(x_1^2-x_2^2x_3=0)\subset 
\mathbb A^{n+1}=\Spec \mathbb C[x_1, \cdots, x_{n+1}]
$$ 
and 
$$
C=(x_1=x_2=0)\subset V\subset \mathbb A^{n+1}. 
$$ 
Let $\varphi:Bl_C\mathbb A^{n+1}\to \mathbb A^{n+1}$ be the blow-up 
whose center is $C$. 
Let $W\simeq C\times \mathbb P^1$ be the exceptional 
divisor of  the above blow-up and let $\pi=\varphi|_W:W\to C$ be the natural 
projection. 
We set $D=V'|_W$ where $V'$ is the strict transform of $V$ on $Bl_C\mathbb A^{n+1}$. 
Assume that $B$ is an $\mathbb R$-Cartier $\mathbb R$-divisor 
on $C$ such that $(D, \pi^*B|_D)$ is 
sub log canonical. 
Then the pair 
$(W, D+\pi^*B)$ is sub log canonical. 

Furthermore, we obtain the following description. 
A closed subset $Q$ of $C$ is the $\pi$-image of some lc center of $(W, D+\pi^*B)$ if and only if 
$Q=C$ or $Q$ is the $\pi|_D$-image of some lc center of $(D, \pi^*B|_D)$. 
\end{lem}
\begin{proof}

We can check that 
$K_W+D=\pi^*(K_V|_C)$ because 
$$K_{Bl_C\mathbb A^{n+1}}+V'+W=\varphi^*(K_{\mathbb A^{n+1}}+V).$$ 
Therefore, $K_W+D+\pi^*B=\pi^*(K_V|_C+B)$. 
Note that 
it is easy to see that $D$ is a smooth divisor on $W$ and that 
$\pi|_D:D\to C$ is a finite morphism with $\deg \pi|_D=2$ which 
ramifies only over $A$, where 
$$
A=(x_1=x_2=x_3=0)\subset C\subset V\subset \mathbb A^{n+1}. 
$$ 
By adjunction, 
$
K_D=(\pi|_D)^*(K_V|_C)$.  
We consider the following base change diagram 
$$
\xymatrix{W\ar[d]_\pi &\ar[l]_q \widetilde W\ar[d]^p\\ 
C &\ar[l]^{\pi|_D}D
}
$$ 
where 
$\widetilde W=W\times _CD$. 
Then we obtain 
$$
K_{\widetilde W}-q^*\left(\frac{1}{2}\pi^*A\right)+q^*D=p^*K_D
$$ 
by $K_W+D=\pi^*(K_V|_C)$ and $K_D=(\pi|_D)^*(K_V|_C)$, 
and have 
\begin{align}\tag{$\heartsuit$}\label{hea}
K_{\widetilde W}-q^*\left(\frac{1}{2}\pi^*A\right)+q^*D+q^*\pi^*B=p^*(K_D+\pi^*B|_D). 
\end{align}
Note that $q^*D=D_1+D_2$ such that 
$D_1$ and $D_2$ are sections of $p:\widetilde W\to D$. 
By the construction, we can check that 
$D_1|_{D_2}=q^*\left(\frac{1}{2}\pi^*A\right)|_{D_2}$ and 
$D_2|_{D_1}=q^*\left(\frac{1}{2}\pi^*A\right)|_{D_1}$. 
We also note that $p$ is smooth and 
$p:D_1\cap D_2\simeq \frac{1}{2}(\pi|_D)^*A$. 
We take a resolution of singularities $\alpha:D^{\dag}\to D$ of the pair 
$(D, \pi^*(A+B)|_D)$, which 
is a finite composite of blow-ups whose centers are smooth. 
We consider the base change of $p:\widetilde W\to 
D$ by $\alpha$. 
$$
\xymatrix{\widetilde W\ar[d]_p&\ar[l] \widetilde W\times _DD^\dag\ar[d]\\ 
D &\ar[l]^{\alpha}D^\dag
}
$$ 
Then $W^{\dag}=\widetilde W\times _D D^{\dag}$ is smooth since $p$ is smooth. 
By the above construction, 
we can easily see that all the discrepancy coefficients of 
$(\widetilde W, -q^*\left(\frac{1}{2}\pi^*A\right)+q^*D+q^*\pi^*B)$ 
are $\geq -1$ since $(D, \pi^*B|_D)$ is sub log canonical and the equation (\ref{hea}) holds. 
Therefore,  
$(\widetilde W, -q^*\left(\frac{1}{2}\pi^*A\right)+q^*D+q^*\pi^*B)$ 
is sub log canonical. 
Since 
$$
K_{\widetilde W}-q^*\left(\frac{1}{2}\pi^*A\right)+q^*D+q^*\pi^*B=q^*(K_W+D+\pi^*B), 
$$ 
we have that $(W, D+\pi^*B)$ has only sub log canonical singularities. 

The description of the $\pi$-images of lc centers of $(W, D+\pi^*B)$ is almost obvious by the above 
discussions. 
\end{proof}

\section{Proofs of the fundamental theorems}\label{sec4}

In this section, we prove the theorems in Section \ref{sec1}. First, 
let us recall Koll\'ar's double covering trick.  

\begin{lem}[A natural double cover due to Koll\'ar]\label{double} 
Let $(X, \Delta)$ be a semi log canonical pair. 
Then we can construct a finite 
morphism $p:\widetilde X\to X$ with the following properties. 
\begin{itemize}
\item[(1)] Let $X^0$ be the largest Zariski open subset whose singularities are double normal crossing 
points only. 
Then $$p^0=p|_{p^{-1}(X^0)}:\widetilde X^0:=p^{-1}(X^0)\to X^0$$ is an \'etale double cover. 
\item[(2)] $\widetilde X$ satisfies Serre's $S_2$ condition, $p$ is \'etale in codimension one, 
the normalization of $\widetilde X$ is a disjoint union of two copies of the normalization of $X$. 
\item[(3)] The irreducible components of $\widetilde X$ are smooth in codimension one. 
\end{itemize} 
In particular, $(\widetilde X, \widetilde \Delta)$ is semi log canonical 
where 
$$ 
K_{\widetilde X}+\widetilde \Delta=p^*(K_X+\Delta). 
$$
\end{lem}
For the construction and related topics, see \cite[5.23]{kollar-book}. 
Let us start the proofs of the fundamental theorems in Section \ref{sec1}. 

\begin{proof}[Proof of {\em{Theorem \ref{vani-thm}}} and {\em{Theorem \ref{vani-thm2}}}] 
It is sufficient to prove Theorem \ref{vani-thm2}. 
This is because Theorem \ref{vani-thm} is a special case of Theorem \ref{vani-thm2}. 
By Lemma \ref{double}, we can take a double cover 
$p:\widetilde X\to X$. 
Since $\mathcal O_X(D)$ is a direct summand of 
$p_*\mathcal O_{\widetilde X}(p^*D)$, 
we may assume that the irreducible components of $X$ are smooth in codimension one by 
replacing $X$ with $\widetilde X$. 
Without loss of generality, we may assume that $S$ is affine by shrinking $S$. 
Therefore, $X$ is quasi-projective. 
By Theorem \ref{main2}, we can construct a 
quasi-log resolution $h:Z\to X$. 
Note that we may assume that $h$ is birational 
by Remark \ref{rem13}. We may further assume that 
$\Supp h^*D\cup \Supp \Delta_Z$ is a simple normal crossing divisor on $Z$ by \cite[Theorem 1.4]{bierstone-p} 
when $D$ is not a Cartier divisor. 
By the construction, 
$$
h^*D+\lceil -\Delta_Z^{<1}\rceil-(K_Z+
\Delta_Z^{=1}+\{\Delta_Z\})\sim _{\mathbb R}h^*(D-(K_X+\Delta)). 
$$ 
If $D$ is Cartier, then 
$$
R^i\pi_*h_*\mathcal O_Z(h^*D+\lceil -\Delta_Z^{<1}\rceil)\simeq R^i\pi_*\mathcal O_X(D)=0$$ 
for every $i>0$ by \cite[Theorem 2.47 or Theorem 3.38]{book}. 
We note that 
$h_*\mathcal O_Z(h^*D+\lceil -\Delta_Z^{<1}\rceil)\simeq \mathcal O_X(D)$. 
From now on, we assume that $D$ is not Cartier. 
Let $\{h^*D\}=\sum _i b_i B_i$ and $\{\Delta_Z\}=\sum _i c_i B_i$ be the irreducible decompositions. 
If $c_i\geq b_i$, then we set $d_i=0$ and $e_i=c_i-b_i\geq 0$. 
If $c_i<b_i$, then we set $d_i=1$ and $e_i=c_i+1-b_i<1$. 
We define $E=\lceil -\Delta_Z^{<1}\rceil+\sum _i d_i B_i$ and $F=\sum _i e_i B_i$. 
Then we have 
$$
\lfloor h^*D\rfloor +E-(K_Z+\Delta_Z^{=1}+F)\sim _{\mathbb R}h^*(D-(K_X+\Delta)). 
$$ 
By the construction, $E$ is an effective $h$-exceptional 
divisor on $X$ and $\{F\}=0$. 
Note that 
$E$ and $\lfloor h^*D\rfloor$ are both Cartier divisors 
on $Z$. 
This is because 
$\Supp h^*D\cup \Supp \Delta_Z$ is a simple 
normal crossing divisor on $Z$ 
and $h^*D$ and $\Delta_Z$ are $\mathbb R$-Cartier $\mathbb R$-divisors on $Z$. 
By \cite[Theorem 2.47 or Theorem 3.38]{book}, we obtain 
that 
$$
R^i\pi_*h_*\mathcal O_Z(\lfloor h^*D\rfloor +E)=0$$ for every $i>0$. 
Therefore, 
$R^i\pi_*\mathcal O_X(D)=0$ for every $i>0$ since
$h_*\mathcal O_Z(\lfloor h^*D\rfloor +E)\simeq \mathcal O_X(D)$. 
\end{proof}

\begin{proof}[Proof of {\em{Theorem \ref{kod-vani}}}] 
We take a Cartier divisor $L$ on $X$ such that 
$\mathcal O_X(L)\simeq \mathcal L$. Without loss of generality, 
we may assume that 
the supports of the Weil divisor $K_X$ and $L$ do not contain any irreducible components of the conductor 
of $X$. 
Since $(K_X+L)-K_X$ is ample by 
the assumption, we obtain $H^i(X, \omega_X\otimes \mathcal L)=0$ for every $i>0$ by 
Theorem \ref{vani-thm}. 
Note that $\omega_X\otimes \mathcal L\simeq \mathcal O_X(K_X+L)$. 
\end{proof}

\begin{proof}[Proof of {\em{Corollary \ref{cor-new}}}] 
Without loss of generality, we may assume that the 
support of the Weil divisor $K_X$ does not contain any irreducible components of the conductor of $X$. 
By the assumption, $mK_X-K_X$ is ample if $m\geq 2$. Therefore, 
we obtain $H^i(X, \mathcal O_X(mK_X))=0$ for every $i>0$ and $m\geq 2$ by Theorem \ref{vani-thm}. 
\end{proof}

\begin{proof}[Proof of {\em{Theorem \ref{vani-thm3}}}]
Since the claim is local, 
we may assume that $S$ is quasi-projective by shrinking 
$S$. 
By Theorem \ref{main2}, $[X, K_X+\Delta]$ has a quasi-log structure induced by 
the semi log canonical structure of $(X, \Delta)$ since 
$X$ is quasi-projective. 
Therefore, $R^i\pi_*(\mathcal I_{X'}\otimes \mathcal O_X(D))=0$ for every 
$i>0$ by \cite[Theorem 3.39]{book}. 
\end{proof}

\begin{rem}\label{rem42}
Let $\{C_i\}_{i\in I}$ be the set of slc strata of $(X, \Delta)$. 
We set 
$$
I_1=\{i\in I \,|\, C_i \subset X'\}
$$ 
and 
$$
I_2=\{i\in I \,|\, C_i\not\subset X'\}. 
$$ 
Then, for the vanishing theorem:~Theorem \ref{vani-thm3}, 
the following weaker assumption is sufficient. 
\begin{itemize}
\item $D-(K_X+\Delta)$ is nef over $S$ and 
$(D-(K_X+\Delta))|_{C_i}$ is big over $S$ for every $i\in I_2$. 
\end{itemize}
It is obvious by the proof given in \cite[Theorem 3.39]{book}. 
\end{rem} 

\begin{proof}[Proof of {\em{Theorem \ref{thm-torsion}}}] 
It is obvious that the claim holds for $\pi_*\mathcal O_X(D)$. 
By Lemma \ref{double}, 
we can take a natural double cover $p:\widetilde X\to X$. 
Since $\mathcal O_X(D)$ is a direct summand of $p_*\mathcal O_{\widetilde X}
(p^*D)$, we may 
assume that the irreducible components of $X$ have no self-intersection in codimension one by replacing 
$X$ with $\widetilde X$. 
Without loss of generality, we may assume that 
$S$ is affine by shrinking $S$. 
Therefore, $X$ is quasi-projective and we can apply Theorem \ref{main2}. 
Let $h:Z\to X$ be a morphism constructed in Theorem \ref{main2}. 
We may assume that  $h$ is birational (see Remark \ref{rem13}). 
Note that 
$$
h^*D+\lceil -\Delta_Z^{<1}\rceil-(K_Z+\{\Delta_Z\}+\Delta_Z^{=1})
\sim _{\mathbb R} h^*(D-(K_X+\Delta)) 
$$ 
is $(\pi\circ h)$-semi-ample.  
As in the proof of Theorem \ref{vani-thm} and Theorem \ref{vani-thm2}, 
we can write 
$$
\lfloor h^*D\rfloor +E-(K_Z+\Delta_Z^{=1}+F)\sim_\mathbb R 
h^*(D-(K_X+\Delta))
$$
when $D$ is not Cartier. 
Therefore, every associated prime of $R^i(\pi\circ h)_*\mathcal O_Z(\lfloor 
h^*D\rfloor+E)$ 
is the generic point of the $\pi$-image of some 
slc stratum of $(X, \Delta)$ for every $i$ 
by Theorem \ref{main2} (5) (see, for example, \cite[Theorem 2.39 (i)]{book} 
and \cite[Theorem 1.1 (i)]{fujino-vanishing}). 
Since 
\begin{align*}
R^1\pi_*\mathcal O_X(D)&\simeq 
R^1\pi_*(h_*\mathcal O_Z(\lfloor h^*D\rfloor +E))\\ &\subset 
R^1(\pi\circ h)_*\mathcal O_Z(\lfloor h^*D\rfloor+E), 
\end{align*}
the claim holds for $R^1\pi_*\mathcal O_X(D)$. 
When $D$ is Cartier, 
it is sufficient to replace $\lfloor h^*D\rfloor+E$ with 
$h^*D+\lceil -\Delta_Z^{<1}\rceil$ in the above arguments. 
Let $A$ be a sufficiently ample general effective Cartier divisor 
on $X$. 
By considering the short exact sequence 
$$
0\to \mathcal O_X(D)\to \mathcal O_X(D+A)\to \mathcal O_A(D+A)\to 0, 
$$
we obtain 
$$
R^i\pi_*\mathcal O_X(D)\simeq R^{i-1}(\pi|_A)_*\mathcal O_A(D+A)
$$ 
for every $i\geq 2$ since $R^i\pi_*\mathcal O_X(D+A)=0$ for $i\geq 1$. 
Note that $$R^i\pi_*\mathcal O_X(D)\simeq 
R^{i-1}(\pi|_{A'})_*\mathcal O_{A'}(D+A')$$ holds for every $i\geq 2$ and every general 
member $A'$ of $|A|$. 
By induction on dimension, 
every associated prime of $R^{i-1}(\pi|_A)_*\mathcal O_A(D+A)$ is the generic point 
of the $\pi|_A$-image of some 
slc stratum of $(A, \Delta|_A)$ for every $i$. 
Note that $(A, \Delta|_A)$ is semi log canonical with 
$(K_X+A+\Delta)|_A=K_A+\Delta|_A$ and that 
$h^*A=h_*^{-1}A$ and $\Supp (h_*^{-1}A+\Delta_Z)$ are simple normal crossing 
divisors on $Z$ since $A$ is general. 
The above statements also hold for any general member 
$A'$ of $|A|$. 
Therefore, the claim holds for $R^i\pi_*\mathcal O_X(D)$, 
which is isomorphic to $R^{i-1}(\pi|_{A'})_*\mathcal 
O_{A'}(D+{A'})$ for every $i\geq 2$ and 
every general member $A'$ of $|A|$. 
It completes the proof. 
\end{proof}

\begin{proof}[Proof of {\em{Theorem \ref{thm-adj}}}] 
By Theorem \ref{main2}, 
$[X, K_X+\Delta]$ has a quasi-log structure. 
Note that $W$ is an slc stratum of $(X, \Delta)$ if and only if 
$W$ is a qlc center of $[X, K_X+\Delta]$ by Theorem \ref{main2} (5). 
Therefore, by adjunction for quasi-log varieties (see, for example, \cite[Theorem 3.39]{book} and 
\cite[Theorem 3.6]{fujino-qlog}), 
$[X', (K_X+\Delta)|_{X'}]$ has a natural quasi-log structure induced by 
the quasi-log structure of $[X, K_X+\Delta]$. 
Since $[X', (K_X+\Delta)|_{X'}]$ is a qlc pair, 
$X'$ is semi-normal (see, for example, \cite[Remark 3.33]{book} and \cite[Remark 3.2]{fujino-qlog}). 
\end{proof}

\begin{proof}[Proof of {\em{Theorem \ref{vani-thm4}}}] 
Without loss of generality, we may assume that $S$ is affine by shrinking $S$. 
Therefore, we may assume that $X$ is quasi-projective 
and $[X, K_X+\Delta]$ has a quasi-log structure by Theorem \ref{main2}. 
By Theorem \ref{thm-adj}, $[X', (K_X+\Delta)|_{X'}]$ has a natural quasi-log structure 
induced by that of $[X, K_X+\Delta]$. 
Therefore, this theorem is a special case of the vanishing 
theorem for quasi-log varieties (see, for example, \cite[Theorem 3.39 (ii)]{book}). 
\end{proof}

\begin{rem}
In Theorems \ref{vani-thm2}, \ref{vani-thm3}, \ref{thm-torsion}, 
\ref{thm-adj}, and \ref{vani-thm4},  
if $(X, \Delta)$ is log canonical, then 
it is sufficient to assume that $\pi$ is {\em{proper}}. 
This is because $(X, \Delta)$ has a natural quasi-log structure 
when $(X, \Delta)$ is log canonical 
(see, for example, \cite[Example 3.42]{book} and \cite[Proposition 3.3]{fujino-qlog}). 
\end{rem}

\begin{proof}[Proof of {\em{Theorem \ref{bpf}}} and {\em{Theorem \ref{bpf2}}}] 
By shrinking $S$, we may assume that 
$S$ is affine and $X$ is quasi-projective. 
Therefore, by applying Theorem \ref{main2}, 
$(X, \Delta)$ has a natural 
quasi-log structure. 
Thus, by \cite[Theorem 3.36]{book} and \cite[Theorem 4.1]{book}, 
we obtain that $\mathcal O_X(mD)$ is $\pi$-generated for every $m\gg 0$. 
\end{proof}

\begin{proof}[Proof of {\em{Theorem \ref{rational}}}]
The proof of \cite[Theorem 15.1]{fujino-fund} works with only minor modifications 
if we adopt Theorem \ref{vani-thm3}. 
We do not need the theory of quasi-log varieties for the proof of the rationality theorem. 
\end{proof}

\begin{proof}[Proof of {\em{Theorem \ref{cone-contraction}}}]
The proof of \cite[Theorem 16.1]{fujino-fund} works with only minor 
modifications by Theorem \ref{rational} and Theorem \ref{bpf}. 
Here we only give a supplementary argument on (1). 
Let $R$ be a $(K_X+\Delta)$-negative extremal ray. 
Then there is a contraction morphism $\varphi_R:X\to Z$ over $S$ associated to $R$ (cf.~(3)). 
Note that $-(K_X+\Delta)$ is $\varphi_R$-ample. 
Let $\nu:X^\nu\to X$ be the normalization. 
We set $K_{X^\nu}+\Theta=\nu^*(K_X+\Delta)$. 
Then $-(K_{X^\nu}+\Theta)$ is $(\varphi_R\circ \nu)$-ample and 
$\varphi_R\circ \nu$ is nontrivial. 
Note that $(X^\nu, \Theta)$ is log canonical. 
By \cite[Theorem 18.2]{fujino-fund}, we can find a rational curve $C'$ on $X^\nu$ such that 
$-(K_{X^\nu}+\Theta)\cdot C'\leq 2\dim X^\nu$ and 
$(\varphi_R\circ \nu)(C')$ is a point. 
We set $C=\nu(C')$. Then 
$C$ is a rational curve on $X$ and $-(K_X+\Delta)\cdot C\leq 2\dim X$ such that 
$\varphi_R(C)$ is a point. 
Therefore, $C$ is a desired curve in $(1)$. 
\end{proof}

We close this section with an important example. 
This example shows that 
we can not always run the minimal model program even for semi log canonical 
{\em{surfaces}}. For some related examples, see \cite{ko-ex}. 
However, Kento Fujita (\cite{fujita2}) establishes a variant of 
the minimal model program for {\em{semi-terminal}} pairs 
in order to construct {\em{semi-terminal modifications}} 
for quasi-projective {\em{demi-normal}} pairs. 
His arguments use not only Theorem \ref{cone-contraction}, 
but also Koll\'ar's gluing theory (see \cite[Section 5]{kollar-book}). 
For the details, see \cite{fujita2}. 

\begin{ex}[{see \cite[Example 3.76]{book}}]\label{ex-bad} We consider 
the first projection $p:\mathbb P^1\times \mathbb P^1\to \mathbb P^1$. 
We take a blow-up $\mu:Z\to \mathbb P^1\times \mathbb P^1$ at 
$(0, \infty)$. 
Let $A_{\infty}$ (resp.~$A_0$) be 
the strict transform of $\mathbb P^1\times \{\infty\}$ 
(resp.~$\mathbb P^1\times \{0\}$) on $Z$. 
We define $M=\mathbb P_Z(\mathcal O_Z\oplus \mathcal O_Z(A_0))$ and 
$X$ is the restriction of $M$ on $(p\circ \mu)^{-1}(0)$. 
Then $X$ is a simple normal crossing divisor on $M$. 
More explicitly, $X$ is a $\mathbb P^1$-bundle 
over $(p\circ \mu)^{-1}(0)$ and is obtained by gluing $X_1=
\mathbb P^1\times \mathbb P^1$ 
and $X_2=
\mathbb P_{\mathbb P^1}(\mathcal O_{\mathbb P^1}\oplus \mathcal 
O_{\mathbb P^1}(1))$ along a fiber. 
In particular, $(X, 0)$ is a semi log canonical surface. 
By the construction, $M\to Z$ has two sections. 
Let $D^+$ (resp.~$D^-$) be the restriction of the section 
of $M\to Z$ corresponding to 
$\mathcal O_Z\oplus \mathcal O_Z(A_0)\to \mathcal O_Z(A_0)\to 0$ 
(resp.~$\mathcal O_Z\oplus \mathcal O_Z(A_0)\to \mathcal O_Z\to 0$). 
Then it is easy to see that $D^+$ is 
a nef Cartier divisor on $X$ 
and that 
the linear system $|mD^+|$ is 
free for every $m>0$. 
Note that $M$ is a projective 
toric variety. 
Let $E$ be the section of $M\to Z$ corresponding 
to $\mathcal O_Z\oplus \mathcal O_Z(A_0)\to 
\mathcal O_Z(A_0)\to 0$. 
Then, it is easy to see that 
$E$ is a nef Cartier divisor 
on $M$. Therefore, the linear system $|E|$ is 
free. In particular, $|D^+|$ is free on $X$ since $D^+=E|_{X}$. So, 
$|mD^+|$ is free for every $m> 0$.
We take a general member $B_0\in |mD^+|$ with 
$m\geq 2$. 
We consider $K_X+B$ with $B=D^-+B_0+B_1+B_2$, 
where $B_1$ and $B_2$ are general fibers of $X_1=\mathbb P^1\times 
\mathbb P^1\subset X$. 
We note that $B_0$ does not intersect 
$D^-$. 
Then $(X, B)$ is an embedded simple normal crossing 
pair. In particular, $(X, B)$ is a semi log canonical surface. 
It is easy to see that there exists only one 
integral curve $C$ on $X_2=
\mathbb P
_{\mathbb P^1}(\mathcal O_{\mathbb P^1}\oplus \mathcal O_{\mathbb P^1}
(1))\subset X$ such that 
$C\cdot (K_X+B)<0$. Note that 
$C$ is nothing but the negative 
section of $X_2=\mathbb P_{\mathbb P^1}(\mathcal O_{\mathbb P^1}
\oplus \mathcal O_{\mathbb P^1}(1))\to \mathbb P^1$. 
We also note that 
$(K_X+B)|_{X_1}$ is ample on $X_1$. 
By the cone theorem (see Theorem \ref{cone-contraction}), 
we obtain 
$$
\overline {NE}(X)=\overline {NE}(X)_{(K_X+B)\geq 0}+\mathbb R_{\geq 0} 
[C]. 
$$ 
By the contraction theorem (see Theorem \ref{cone-contraction}), we have $\varphi:X\to W$ which 
contracts $C$. We can easily see that 
$K_W+B_W$, where $B_W=\varphi_*B$, is not $\mathbb Q$-Cartier 
because $C$ is not $\mathbb Q$-Cartier on $X$. 
Therefore, we can not always run the minimal model program for 
semi log canonical surfaces. 
\end{ex}

For a new framework of the minimal model program for {\em{log surfaces}}, 
see \cite{fujino-surface}, \cite{fujino-tanaka}, \cite{tanaka}, and \cite{tanaka2}. 

\section{Miscellaneous applications}\label{sec-mis}

In this paper, we adopt the following definition of 
{\em{stable pairs}}. 
It is a generalization of the notion of {\em{stable pointed curves}}. 

\begin{defn}[Stable pairs]\label{def-stable}
Let $(X, \Delta)$ be a projective semi log canonical 
pair such that $K_X+\Delta$ is ample. 
Then we call $(X, \Delta)$ a {\em{stable pair}}. 
\end{defn}

Stable pairs will play important roles in the theory of moduli of canonically polarized 
varieties. 

\subsection{Base point free theorems revisited}
First, we prove the base point free theorem for $\mathbb R$-divisors 
(see \cite[Theorem 17.1]{fujino-fund}). 
It is an easy consequence of the base point free theorem (see Theorem \ref{bpf}) and 
the cone theorem (see Theorem \ref{cone-contraction}). For the definition and basic properties of 
$\pi$-semi-ample $\mathbb R$-Cartier 
$\mathbb R$-divisors, see, for example, \cite[Definition 4.11, Lemmas 4.13 and 4.14]{fujino-fund}. 

\begin{thm}[Base point free theorem for $\mathbb R$-divisors] 
Let $(X, \Delta)$ be a semi log canonical pair and 
let $\pi:X\to S$ be a projective morphism 
onto an algebraic variety $S$. 
Let $D$ be a $\pi$-nef $\mathbb R$-Cartier $\mathbb R$-divisor 
on $X$. 
Assume that $D-(K_X+\Delta)$ is $\pi$-ample. 
Then $D$ is $\pi$-semi-ample. 
\end{thm}
\begin{proof}
This theorem is an easy consequence of Theorem \ref{bpf} and 
Theorem \ref{cone-contraction} (2). For the details, see the proof of 
\cite[Theorem 17.1]{fujino-fund}. 
\end{proof}

Next, we discuss a generalization of Koll\'ar's 
effective base point free theorem for semi log canonical pairs (cf.~\cite{fujino-eff}). 

\begin{thm}[Effective base point free theorem]\label{63}
Let $(X, \Delta)$ be a projective semi log canonical pair such that $\Delta$ is a $\mathbb Q$-divisor and 
let $L$ be a nef Cartier divisor on $X$. 
Assume that $aL-(K_X+\Delta)$ is nef and log big with 
respect to $(X, \Delta)$ for some real number $a>0$. 
Then there exists a positive integer $m=m(n, a)$, which only depends on 
$n=\dim X$ and $a$, such that $|mL|$ is free. 
\end{thm}

\begin{rem}
We can take $m(n, a)=2^{n+1}(n+1)!(\lceil a\rceil+n)$ in Theorem \ref{63}. For the 
details, see \cite{fujino-eff}. 
\end{rem}

We give a remark on \cite{fujino-eff}. 

\begin{rem} 
In this remark, we use the same notation as in \cite[2.1.1]{fujino-eff}. 
By the vanishing theorem \cite[Theorem 3.2 (b)]{fujino-eff}, we have 
$$
h^i(S, h_*\mathcal O_Y(N'-F))=h^i(S, h_*\mathcal O_Y(N'))=0
$$ 
for all $i>0$. This implies that 
$$
h^i(S, h_*\mathcal O_F(N'))=0
$$ 
for all $i>0$. Therefore, we do not need the vanishing theorem \cite[Theorem 3.2 (b)]{fujino-eff} for a simple 
normal crossing variety $F$. The vanishing theorem for $Y$ is sufficient. Note that 
$Y$ is a smooth variety. 
The vanishing theorem \cite[Theorem 3.2 (b)]{fujino-eff} is much simpler for 
smooth varieties than for simple normal crossing varieties. 
\end{rem}

\begin{proof}[Sketch of the proof of {\em{Theorem \ref{63}}}] 
The arguments in \cite{fujino-eff} work for our situation by suitable modifications. 
We use a quasi-log resolution constructed in Theorem \ref{main2} instead of taking a resolution of singularities 
(cf.~\cite[2.1.1]{fujino-eff}). 
We also use the vanishing theorem for simple normal crossing pairs (see, for example, 
\cite[Theorem 2.39]{book} or \cite[Theorem 1.1]{fujino-vanishing}) and Theorem \ref{vani-thm3}. 
All the other modifications we need are more or less routine works. 
We leave the details for the reader's exercise.  
\end{proof}

\subsection{Shokurov's polytope}
Let us introduce the notion of Shokurov's polytope for semi log canonical pairs. 
It is useful for reducing the problems for $\mathbb R$-divisors to ones for $\mathbb Q$-divisors. 

\begin{say}[Shokurov's polytope]\label{say-sho} 
Let $X$ be an equidimensional algebraic variety which satisfies Serre's $S_2$ condition and is normal crossing 
in codimension one. Let $B$ be a reduced Weil divisor on $X$ whose support does not contain any irreducible 
components of the conductor of $X$. Let $B=\sum _iB_i$ be the irreducible decomposition. 
We define a finite-dimensional $\mathbb R$-vector space $V=\bigoplus _i \mathbb R B_i$. 
Then it is easy to see that 
$$
\mathcal L=\{D\in V\, |\, (X, D) \ \text{is semi log canonical}\}
$$ 
is a rational polytope in $V$. 
Let $\pi:X\to S$ be a projective morphism onto an algebraic variety $S$. 
We can also check that 
$$
\mathcal N=\{ D\in \mathcal L\, | \, K_X+D\ \text{is $\pi$-nef} \}
$$
is a rational polytope (see, for example, 
the proof of \cite[Proposition 3.2]{birkar}). 
A key point is the boundedness of lengths of extremal rays in Theorem \ref{cone-contraction} (1). 
We note that $\mathcal N$ is known as Shokurov's polytope when $X$ is normal. 
Assume that $\Delta$ is an $\mathbb R$-divisor on $X$ such that 
$\Supp \Delta\subset \Supp B$, $(X, \Delta)$ is semi log canonical, and 
$K_X+\Delta$ is $\pi$-nef. 
Then $\Delta\in \mathcal N$. 
In this case, 
we can write 
$$
K_X+\Delta=\sum _{i=1}^k r_i (K_X+D_i)
$$ 
such that 
\begin{itemize}
\item[(i)] $D_i\in \mathcal N$ is an effective $\mathbb Q$-divisor on $X$ for every $i$, 
\item[(ii)] $(X, D_i)$ is semi log canonical for every $i$, and 
\item[(iii)] $0<r_i<1$, $r_i\in \mathbb R$ for every $i$, and $\sum _{i=1}^kr_i=1$. 
\end{itemize}
If $\Delta$ is contained in a face $\mathcal F$ of $\mathcal N$, then 
we can choose $D_i$ such that $D_i\in \mathcal F$ for every $i$. 
Moreover, we can make $D_i$ arbitrarily close to $\Delta$ in a given norm on $V$ for every 
$i$. 
\end{say}

The abundance conjecture is one of the most important conjectures in the 
minimal model theory. 

\begin{conj}[(Log) abundance conjecture]\label{abun-conj}
Let $(X, \Delta)$ be a semi log canonical pair and let $\pi:X\to $S be a projective 
morphism. Suppose that $K_X+\Delta$ is $\pi$-nef. Then $K_X+\Delta$ is $\pi$-semi-ample. 
\end{conj}

By the arguments in \ref{say-sho}, we may assume that $\Delta$ is a $\mathbb Q$-divisor on $X$. 
This reduction seems to be very important because 
we do not know how to use the gluing 
arguments for $\mathbb R$-divisors (cf.~\cite{fujino-abun}, 
\cite{fujino-gongyo}, 
\cite{hacon-xu}). 
We note that if $\Delta$ is a $\mathbb Q$-divisor then 
Conjecture \ref{abun-conj} can be reduced to the case when $(X, \Delta)$ is 
log canonical, that is, $X$ is normal (cf.~\cite{fujino-gongyo}, \cite{hacon-xu}). 

From now on, we treat the two extreme cases of Conjecture \ref{abun-conj}. 

\begin{thm}[Numerically trivial case]
Let $(X, \Delta)$ be a semi log canonical pair and let $\pi:X\to S$ be 
a projective morphism onto an algebraic variety $S$. 
Assume that $K_X+\Delta$ is numerically trivial over $S$. Then 
$K_X+\Delta$ is $\pi$-semi-ample. 
\end{thm}
\begin{proof}
We set $B=\lceil \Delta\rceil$ and apply the arguments in \ref{say-sho}. 
Then we can write 
$$
K_X+\Delta=\sum _{i=1}^k r_i (K_X+D_i)
$$ 
as in \ref{say-sho}. 
Since $K_X+\Delta$ is numerically $\pi$-trivial and 
$K_X+D_i$ is $\pi$-nef for every $i$, $K_X+D_i$ is numerically $\pi$-trivial for 
every $i$. 
Therefore, we can reduce the problem to the case when $\Delta$ is a $\mathbb Q$-divisor. 
If $\Delta$ is a $\mathbb Q$-divisor, then the statement is nothing but 
\cite[Corollary 4.11]{fujino-gongyo} (see also \cite[Subsection 4.1]{fujino-gongyo}). 
Therefore, $K_X+\Delta$ is $\pi$-semi-ample. 
\end{proof}

\begin{thm}[Nef and log big case]
Let $(X, \Delta)$ be a semi log canonical pair and let 
$\pi:X\to S$ be a projective morphism onto an algebraic 
variety $S$. 
Assume that $K_X+\Delta$ is nef and log big over $S$ with respect to $(X, \Delta)$. 
Then $K_X+\Delta$ is $\pi$-semi-ample. 
\end{thm}

\begin{proof}
We set $B=\lceil \Delta\rceil$ and apply the arguments in \ref{say-sho}. 
Then we can write 
$$
K_X+\Delta=\sum _{i=1}^k r_i (K_X+D_i)
$$ 
as in \ref{say-sho}. 
If $D_i$ is sufficiently close to $\Delta$, 
then $K_X+D_i$ is nef and log big over $S$ with respect to $(X, \Delta)$. 
This is because the bigness is an open condition. 
It is easy to see that $K_X+D_i$ is nef and log big over $S$ with respect to $(X, D_i)$ if $D_i$ is sufficiently 
close to $\Delta$. 
Therefore, we may assume that $\Delta$ is a $\mathbb Q$-divisor. 
In this case, we can check that $K_X+\Delta$ is semi-ample over $S$ by 
Theorem \ref{bpf2}. 
\end{proof}

\subsection{Depth of sheaves on slc pairs} 
The following theorem is an $\mathbb R$-divisor version of Koll\'ar's 
result (see \cite[Theorem 3]{kollar-local}), which is 
a generalization of \cite[Lemma 3.2]{alexeev-limits} and \cite[Theorem 4.21]{book}. 
It can be proved by the method of two spectral sequences of local 
cohomology groups (cf.~\cite[4.2.1 Appendix and 
Section 4.3]{book}). 
For the details and some interesting examples, see \cite{kollar-local}. 
For some related topics, see \cite{kovacs-irr} and \cite{alexeev-hacon}. 

\begin{thm}\label{thm-depth} 
Let $(X, \Delta)$ be a semi log canonical pair and let $x\in X$ be a scheme theoretic point. 
Assume that $x$ is not the generic point of any slc center of $(X, \Delta)$. 
Then we have the following properties. 

\begin{itemize}
\item[(1)] Let $D$ be a Weil divisor on $X$ 
whose support does not contain any irreducible components of the 
conductor of $X$. 
Let $\Delta'$ be an effective $\mathbb R$-divisor on $X$ such that 
$\Delta'\leq \Delta$ and that $D\sim _{\mathbb R, loc}\Delta'$, that is, 
$D$ is locally $\mathbb R$-linearly equivalent to $\Delta'$. Then 
$$
\mathrm{depth}_x\mathcal O_X(-D)\geq \min\{3, \mathrm{codim}_{X} x\}. 
$$
\item[(2)] Let $X'$ be any reduced closed subscheme of $X$ that is 
a union of some slc centers of $(X, \Delta)$. 
Then 
$$
\mathrm{depth}_x \mathcal I_{X'}\geq \min\{3, 1+\mathrm{codim}_{X'} x\}, 
$$ 
where $\mathcal I_{X'}$ is the defining ideal sheaf of $X'$ on $X$. 
\end{itemize}
\end{thm}
\begin{proof}[Sketch of the proof of {\em{Theorem \ref{thm-depth}}}] 
First, we consider (2). The proof of Theorem 3 (2) in \cite{kollar-local} works 
without any changes. Next we consider (1). 
Since the problem is local, we may assume that $X$ is affine and $D\sim _{\mathbb R}\Delta'$. 
By considering the real vector space spanned by the irreducible 
components of $\Supp \Delta$ and perturbing $\Delta$ and 
$\Delta'$, we can find effective $\mathbb Q$-divisors 
$\Delta'_0$ and $\Delta_0$ on $X$ such that $\Delta'_0\leq \Delta_0$, 
$D\sim _{\mathbb Q}\Delta'_0$, $(X, \Delta_0)$ is semi log canonical, and $x$ is not the 
generic point of any slc center of $(X, \Delta_0)$. 
Therefore, by Theorem 3 (1) in \cite{kollar-local}, 
we obtain the desired inequality. 
\end{proof}

\subsection{Slc morphisms} 
In this subsection, we introduce the notion of {\em{slc morphisms}} and 
prove some basic properties. 

\begin{defn}[Slc morphisms]\label{slc-mor}
Let $(X, \Delta)$ be a semi log canonical pair and let $f:X\to C$ be a flat morphism 
onto a smooth curve $C$. 
We say that $f:(X, \Delta)\to C$ is {\em{semi log canonical}} ({\em{slc}}, for short) if 
$(X, \Delta+f^*c)$ is semi log canonical for every closed point $c\in C$.  
\end{defn}

The following lemma is almost obvious by the definition of slc morphisms. 
See \cite[Lemma 7.2]{km}. 
\begin{lem}\label{lem69}
Assume that $f:(X, \Delta)\to C$ is slc. Then we have the following properties. 
\begin{itemize}
\item[(1)] Every fiber of $f$ is reduced. 
\item[(2)] $\Delta$ is horizontal, that is, none of the irreducible components of 
$\Delta$ is contained in a fiber of $f$. 
\item[(3)] If $E$ is an exceptional divisor over $X$ such that 
the center $c_X(E)$ is contained in a fiber, then $a(E, X, \Delta)\geq 0$. 
\end{itemize}
\end{lem}

By the same arguments as in the proof of \cite[Lemma 7.6]{km}, we 
know that the notion of slc morphisms is stable under base changes. 

\begin{lem}
Assume that $f:(X, \Delta)\to C$ is slc. 
Let $g:C'\to C$ be a non-constant morphism from a smooth curve 
$C'$, $X'=X\times _CC'$ with 
projections $h:X'\to X$ and $f':X'\to C'$. 
We set $K_{X'}+\Delta'=h^*(K_X+\Delta)$. 
Then $f':(X', \Delta')\to C'$ is also slc. 
\end{lem}

The following theorem is the main result of this subsection. 
It is a consequence of Theorem \ref{thm-torsion}. 

\begin{thm}
Let $f:X\to C$ be a projective semi log canonical 
morphism. 
Then $R^if_*\mathcal O_X(K_X)$ is locally free for every $i$. 
Therefore, for every $i$, 
we obtain that $R^if_*\mathcal O_X(K_{X/C})$ is locally free and that 
$$R^if_*\mathcal O_X(K_{X/C})\otimes \mathbb C(c)
\simeq H^i(X_c, \mathcal O_{X_c}(K_{X_c}))$$ for all $c\in C$, 
where $X_c=f^{-1}(c)$. 
In particular, $\dim _{\mathbb C} H^i(X_c, \mathcal O_{X_c}(K_{X_c}))$ is independent of $c\in C$. 
\end{thm}

\begin{proof}By Theorem \ref{thm-torsion}, $R^if_*\mathcal O_X(K_X)$ is torsion-free because 
every slc stratum of $X$ is dominant onto $C$ (see Lemma \ref{lem69} (3)). 
The other claims are obvious by the base change theorem (cf.~\cite[(4.3)]{kollar-local}). 
\end{proof}

\subsection{Finiteness of birational automorphisms}
This subsection is a supplement to \cite{fujino-gongyo}. 
Let us introduce the notion of {\em{$B$-birational maps}} for semi log canonical pairs 
(cf.~\cite{fujino-abun}, \cite{fujino-gongyo}). 
 
\begin{defn}[{cf.~\cite[Definition 3.1]{fujino-abun}, \cite[Definition 2.11]{fujino-gongyo}}]
Let $(X,\Delta)$ be a semi log canonical pair. 
We say that a proper birational map $f:(X,\Delta)\dashrightarrow (X,\Delta)$ 
is {\em {$B$-birational}} if there exists a common resolution 
\begin{equation*}
\xymatrix{ & W\ar[dl]_{\alpha} \ar[dr]^{\beta}\\
 X \ar@{-->}[rr]_{f}  & & X}
\end{equation*}
such that $$\alpha^*(K_X+\Delta)=\beta^*(K_X+\Delta). $$
This means that it holds that $E_\alpha=E_\beta$ when we write 
$$K_W=\alpha^*(K_X +\Delta)+E_\alpha$$ 
and $$K_W=\beta^*(K_X+\Delta)+E_\beta.$$ 
We define  
$$\Bir(X,\Delta)=\{f\,|\, f :(X,\Delta) 
\dashrightarrow (X,\Delta) \text{ is $B$-birational} \}. $$ 
It is obvious that $\Bir (X, \Delta)$ has a natural group structure. 
We also define 
$$
\Aut (X, \Delta)=\{f\, |\, f:X\to X \ {\text{is an isomorphism such that}} \ 
\Delta=f_*^{-1}\Delta\}. 
$$
We can easily see that $\Aut (X, \Delta)$ is a subgroup of $\Bir (X, \Delta)$. 
\end{defn}

The following theorem is the main theorem of this subsection. 
It is essentially contained in \cite[Corollary 3.13]{fujino-gongyo}.  

\begin{thm}[Finiteness of $B$-birational maps]\label{67} 
Let $(X, \Delta)$ be a complete semi log canonical 
pair such that $\Delta$ is a $\mathbb Q$-divisor. 
Assume that $K_X+\Delta$ is big, that is, 
$K_{X_i^\nu}+\Theta_i$ is big for every $i$, 
where $\nu:X^\nu\to X$ is the normalization, 
$X^\nu=\cup _i X_i^\nu$ is the irreducible 
decomposition, and $K_{X_i^\nu}+\Theta_i=\nu^*(K_X+\Delta)|_{X_i^\nu}$. 
Then $\Bir (X, \Delta)$ is a finite group. 
In particular, $\Aut (X, \Delta)$ is a finite group. 
\end{thm}
\begin{proof}
Let $f:Y\to X$ be a resolution such that 
$Y$ is projective, $K_Y+\Delta_Y=f^*(K_X+\Delta)$, 
and $\Supp \Delta_Y$ is a simple normal crossing divisor on $Y$. 
It is easy to see that 
$\Bir (X, \Delta)$ is isomorphic to $\Bir (Y, \Delta_Y)$ because 
$f$ is birational. 
By \cite[Corollary 3.13 and Remark 3.16]{fujino-gongyo}, 
we know that $\Bir (Y, \Delta_Y)$ is a finite group. 
Therefore, so is $\Bir (X, \Delta)$. 
Since $\Aut (X, \Delta)$ is a subgroup of $\Bir (X, \Delta)$, 
$\Aut (X, \Delta)$ is also a finite group. 
\end{proof}

As a direct consequence of Theorem \ref{67}, we obtain the following corollary. 

\begin{cor}\label{cor614}
Let $(X, \Delta)$ be a stable pair such that $\Delta$ is a $\mathbb Q$-divisor. 
Then $\Bir (X, \Delta)$ and $\Aut (X, \Delta)$ are finite groups. 
\end{cor}

\begin{proof}
Since $K_{X^\nu}+\Theta$ is ample, where 
$\nu:X^\nu\to X$ is the normalization and $K_{X^\nu}+\Theta=\nu^*(K_X+\Delta)$, 
$\Bir (X, \Delta)$ is a finite group by Theorem \ref{67}. 
Therefore, so is $\Aut (X, \Delta)$. 
\end{proof}

Corollary \ref{cor614} seems to be indispensable 
when we consider moduli problems for stable pairs. 

\section{Appendix:~Big $\mathbb R$-divisors}\label{sec-appendix}

In this appendix, we discuss the notion of 
big $\mathbb R$-divisors on singular 
varieties. The basic references of big $\mathbb R$-divisors are 
\cite[2.2]{lposi} and 
\cite[II.~\S 3 and \S5]{nakayama2}. 
Since we have to consider big $\mathbb R$-divisors on 
{\em{non-normal}} 
varieties, we give supplementary 
definitions and arguments to \cite{lposi} and \cite{nakayama2}. 

First, let us quickly recall the definition of 
big Cartier divisors on normal complete irreducible varieties. 
For details, see, for example, \cite[\S 0-3]{kmm}. 

\begin{defn}[Big Cartier divisors]\label{def-big}
Let $X$ be a normal complete irreducible variety and let $D$ be a Cartier 
divisor 
on $X$. 
Then $D$ is {\em{big}} if one of the 
following equivalent conditions holds. 
\begin{itemize}
\item[(1)] $\underset{m\in \mathbb N}{\max}\{ \dim \Phi_{|mD|}(X)\}=\dim X$, 
where $\Phi _{|mD|}:X\dashrightarrow \mathbb P^N$ is the 
rational map associated to 
the linear system $|mD|$ and $\Phi_{|mD|}(X)$ is the 
image of $\Phi _{|mD|}$. 
\item[(2)] There exist a rational number $\alpha$ and a positive 
integer $m_0$ such that 
$$
\alpha m^{\dim X}\leq \dim H^0(X, \mathcal O_X(mm_0D))
$$ 
for every $m\gg 0$. 
\end{itemize}
It is well known that 
we can take $m_0=1$ in the condition (2). 
\end{defn}
One of the most important properties of big Cartier divisors 
is known as Kodaira's lemma. 

\begin{lem}[Kodaira's lemma]\label{kod-lem} 
Let $X$ be a normal complete irreducible variety and 
let $D$ be a big Cartier divisor on $X$. 
Then, for an arbitrary Cartier divisor $M$, 
we have $H^0(X, \mathcal O_X(lD-M))\ne 0$ for every $l \gg 0$. 
\end{lem}
\begin{proof}
By replacing $X$ with its resolution, 
we may assume that $X$ is smooth and 
projective. 
Then it is sufficient to show that 
for a sufficiently ample Cartier divisor 
$A$, 
$H^0(X, \mathcal O_X(lD-A))\ne 0$ for every $l \gg 0$. 
Since we have the exact sequence 
$$
0\to \mathcal O_X(lD-A)\to \mathcal O_X(lD)\to \mathcal O_Y(lD)\to 0, 
$$ 
where $Y$ is a general member of $|A|$, 
and since there exist positive rational 
numbers $\alpha$, $\beta$ such that 
$\alpha l ^{\dim X}\leq \dim H^0(X, \mathcal O_X(lD))$ and 
$\dim H^0(Y, \mathcal O_Y(lD))\leq {\beta} l ^{\dim Y}$ for 
every $l \gg 0$, 
we have $H^0(X, \mathcal O_X(lD-A))\ne 0$ for every $l \gg 0$. 
\end{proof}

For non-normal varieties, we need the following definition. 

\begin{defn}[Big Cartier divisors on non-normal 
varieties]\label{def-big2}
Let $X$ be a complete irreducible variety and 
let $D$ be a Cartier divisor on $X$. 
Then $D$ is {\em{big}} 
if 
$\nu^*D$ is big on $X^\nu$, where $\nu:X^\nu \to X$ is 
the normalization. 
\end{defn}

Before we define big $\mathbb R$-divisors, let us recall the 
definition of big $\mathbb Q$-divisors. 

\begin{defn}[Big $\mathbb Q$-divisors]\label{defnQ}
Let $X$ be a complete irreducible variety 
and let $D$ be a $\mathbb Q$-Cartier $\mathbb Q$-divisor 
on $X$. 
Then $D$ is {\em{big}} if $mD$ is a big Cartier divisor for 
some positive integer $m$. 
\end{defn}

We note the following obvious lemma. 

\begin{lem}\label{lem0555}
Let $f:W\to V$ be a birational morphism between normal 
complete irreducible 
varieties and let $D$ be a $\mathbb Q$-Cartier $\mathbb Q$-divisor 
on $V$. 
Then $D$ is big if and only if so is $f^*D$. 
\end{lem}

Next, let us start to consider big $\mathbb R$-divisors. 

\begin{defn}[Big $\mathbb R$-divisors on complete varieties]\label{defnA}
An $\mathbb R$-Cartier $\mathbb R$-divisor $D$ on a complete 
irreducible variety $X$ is {\em{big}} 
if 
it can be written in the form 
$$
D=\sum _i a_i D_i
$$
where each $D_i$ is a big Cartier divisor 
and $a_i$ is a positive real number for every $i$.
\end{defn} 

Let us recall an easy but very important lemma. 

\begin{lem}[{see \cite[2.11. Lemma]{nakayama2}}]\label{lemABC}  
Let $f:Y\to X$ be a proper surjective 
morphism between normal irreducible varieties 
with connected fibers. 
Let $D$ be an $\mathbb R$-Cartier $\mathbb R$-divisor 
on $X$. Then we have a canonical isomorphism 
$$\mathcal O_X(\lfloor 
D\rfloor)\simeq 
f_*\mathcal O_Y(\lfloor f^*D\rfloor). $$
\end{lem}

\begin{lem}\label{lem088} 
Let $D$ be a big $\mathbb R$-Cartier
$\mathbb R$-divisor on a {\em{smooth}} projective irreducible variety $X$.
Then there exist a positive rational number $\alpha$ and
a positive integer $m_0$ such that
$$
\alpha m^{\dim X}\leq \dim H^0(X, \mathcal O_X(\lfloor mm_0D\rfloor))
$$
for every $m \gg 0$.
\end{lem}
 
\begin{proof}
By using Lemma \ref{kod-lem},
we can find an effective $\mathbb R$-Cartier $\mathbb R$-divisor
$E$ on $X$ such that
$D-E$ is ample.
Therefore, there exists a positive integer
$m_0$ such that $A=\lfloor m_0D-m_0E\rfloor$ is ample.
We note that $m_0D=A+\{m_0D-m_0E\}+m_0E$.
This implies that $mA\leq mm_0D$ for any positive
integer $m$.
Therefore,
$$
\dim H^0(X, \mathcal O_X(mA))\leq \dim H^0(X, \mathcal O_X(\lfloor
mm_0D\rfloor)).
$$
So, we can find a positive rational number $\alpha$ such that
$$
\alpha m^{\dim X}\leq \dim H^0(X, \mathcal O_X(\lfloor
mm_0D\rfloor)).
$$ 
It is the desired inequality. 
\end{proof}

\begin{rem}\label{rem099} 
By Lemma \ref{lem0555} and Lemma \ref{lem088}, 
Definition \ref{defnA} is 
compatible with Definition \ref{defnQ}. 
\end{rem}

\begin{lem}[Weak Kodaira's lemma]\label{wkodaira}
Let $X$ be a projective irreducible variety and let $D$ be 
a big $\mathbb R$-Cartier $\mathbb R$-divisor on $X$. 
Then we can write 
$$
D\sim _{\mathbb R}A+E, 
$$
where $A$ is an ample $\mathbb Q$-divisor on $X$ and 
$E$ is an effective $\mathbb R$-Cartier 
$\mathbb R$-divisor on $X$. 
\end{lem}
\begin{proof}
Let $B$ be a big Cartier divisor on $X$ and 
let $H$ be a general very ample Cartier divisor 
on $X$. 
We consider the short exact sequence 
$$
0\to \mathcal O_X(lB-H)\to \mathcal O_X(lB)\to \mathcal O_H(lB)\to 0
$$ 
for every $l$. 
It is easy to see that  
$\dim H^0(X, \mathcal O_X(lB))\geq \alpha l^{\dim X}$ and 
$\dim H^0(H, \mathcal O_H(lB))\leq \beta l^{\dim H}$ for 
some positive rational numbers $\alpha$, $\beta$, and for 
every $l\gg 0$. 
Therefore, $H^0(X, \mathcal O_X(lB-H))\ne 0$ for some 
large $l$. This means that 
$lB\sim H+G$ for some effective Cartier divisor $G$. By Definition \ref{defnA}, 
we can write $D=\sum _i a_i D_i$ 
where $a_i$ is a positive real number and 
$D_i$ is a big Cartier divisor for every $i$. 
By applying the above argument to 
each $D_i$, 
we can easily obtain the desired 
decomposition $D\sim _{\mathbb R}A+E$. 
\end{proof}

We prepare an important lemma. 

\begin{lem}\label{lemN}
Let $X$ be a complete irreducible variety 
and let $N$ be a numerically trivial 
$\mathbb R$-Cartier 
$\mathbb R$-divisor 
on $X$. 
Then $N$ can be written 
in the form 
$$
N=\sum _i r_i N_i
$$
where each $N_i$ is a numerically trivial 
Cartier divisor 
and $r_i$ is a real number for every $i$. 
\end{lem}
\begin{proof}
Let $Z_j$ be an integral $1$-cycle on $X$ for $1\leq j\leq \rho=\rho(X)$ such 
that $\{[Z_1], \cdots, [Z_\rho]\}$ is a basis of the vector space 
$N_1(X)$. 
The condition that an $\mathbb R$-Cartier 
$\mathbb R$-divisor $B=\sum _i b_i B_i$, 
where $b_i$ is a real number and 
$B_i$ is Cartier for every $i$, 
is numerically trivial is given by 
the integer linear equations 
$$
\sum _i b_i (B_i\cdot Z_j)=0
$$ 
on $b_i$ for $1\leq j \leq \rho$. 
Any real solution to these 
equations is an $\mathbb R$-linear combination of integral 
ones. Thus, we obtain the desired expression $N=\sum _i r_i N_i$. 
\end{proof}

The following proposition seems to be very important. 

\begin{prop}\label{lemD}
Let $X$ be a complete irreducible 
variety. 
Let $D$ and $D'$ be $\mathbb R$-Cartier 
$\mathbb R$-divisors on $X$. If $D\equiv D'$, 
then $D$ is big if and only if 
so is $D'$. 
\end{prop}
\begin{proof}
We set $N=D'-D$. Then 
$N$ is a numerically trivial 
$\mathbb R$-Cartier 
$\mathbb R$-divisor 
on $X$. 
By Lemma \ref{lemN}, we can write 
$N=\sum _i r_i N_i$, where 
$r_i$ is a real number and $N_i$ is a numerically trivial Cartier 
divisor for every $i$. 
By Definition \ref{defnA}, 
we are reduced to showing that if $B$ is a big Cartier 
divisor and $G$ is a numerically trivial Cartier divisor, then 
$B+rG$ is big 
for any real number $r$. 
If $r$ is not a rational number, we can 
write 
$$
B+rG=t(B+r_1G)+(1-t)(B+r_2G)
$$
where $r_1$ and $r_2$ are rational, 
$r_1<r<r_2$, and $t$ is a real number with $0<t<1$. 
Therefore, we may assume that $r$ is rational. 
Let 
$f:Y\to X$ be a resolution. 
Then it is sufficient to check that $f^*B+rf^*G$ is big by Lemma \ref{lem0555} and 
Definitions \ref{def-big2}. 
So, we may assume that 
$X$ is smooth and projective. 
By Kodaira's lemma (see Lemma \ref{kod-lem}), we can write 
$lB\sim A+E$, where 
$A$ is an ample Cartier divisor, 
$E$ is an effective Cartier divisor, and $l$ is a 
positive integer. 
Thus, $l(B+rG)\sim (A+lrG)+E$. 
We note that 
$A+lrG$ is an ample $\mathbb Q$-divisor. 
This implies that $B+rG$ is a big $\mathbb Q$-Cartier 
$\mathbb Q$-divisor. 
We finish the proof. 
\end{proof}

We give a small remark on Iitaka's $D$-dimension for $\mathbb R$-divisors. 
Please compare it with Proposition \ref{lemD}. 

\begin{rem}
We consider $X=\mathbb P^1$ and take $P, Q \in X$ with $P\ne Q$. 
We set $D=\sqrt{2}P-\sqrt{2}Q$. 
Then it is obvious that 
$D\sim _{\mathbb R}0$. 
However, 
$\kappa (X, D)=-\infty$ because 
$\deg \lfloor mD\rfloor<0$ for every positive integer $m$. 
Note that $\mathbb R$-linear equivalence does not always preserve 
Iitaka's $D$-dimension. 
\end{rem}

Proposition \ref{lemCD} seems to be missing in the literature. 
We note that $X$ is not assumed to be {\em{projective}} in Proposition \ref{lemCD}. 

\begin{prop}\label{lemCD}
Let $D$ be an $\mathbb R$-Cartier $\mathbb R$-divisor on a 
{\em{normal}} complete irreducible variety $X$.  
Then the following conditions are equivalent. 
\begin{itemize}
\item[(1)] $D$ is big. 
\item[(2)] There exist a positive 
rational number $\alpha$ and a positive 
integer $m_0$ such that 
$$
\alpha m^{\dim X}\leq \dim H^0(X, \mathcal O_X(\lfloor mm_0D\rfloor))
$$ for every $m \gg 0$. 
\end{itemize}
\end{prop}

\begin{proof}First, we assume (2). 
Let $f:Y\to X$ be a resolution such that 
$Y$ is projective. By Lemma \ref{lemABC}, 
we have 
$$\alpha m^{\dim X}\leq \dim H^0(X, \mathcal O_X(\lfloor mm_0 f^*D\rfloor)).$$
By the usual argument as in the proof of Kodaira's lemma (cf.~Lemma \ref{kod-lem}),
we can write $f^*D\equiv A+E$, where $A$ is an ample $\mathbb Q$-Cartier
$\mathbb Q$-divisor and $E$ is an effective $\mathbb R$-Cartier 
$\mathbb R$-divisor on $Y$. 
By using Lemma \ref{lemCDDD} 
below,
we can write $A+E\equiv \sum a_i G_i$ where $a_i$ is a positive real number and $G_i$ is a 
big Cartier divisor for every $i$. By Proposition \ref{lemD}, $f^*D$ is a big
$\mathbb R$-Cartier $\mathbb R$-divisor on $Y$.
Let $D'$ be a $\mathbb Q$-Cartier $\mathbb Q$-divisor on $X$ whose coefficients are very close to those of $D$. Then $A+f^*D'-f^*D$ is an ample
$\mathbb R$-Cartier $\mathbb R$-divisor on $Y$.
Therefore, $f^*D'\equiv (A+f^*D'-f^*D)+E$ is also a big 
$\mathbb Q$-divisor on $Y$ as above.  By
Lemma \ref{lem0555}, $D'$ is a big $\mathbb Q$-Cartier 
$\mathbb Q$-divisor on $X$. 
This means that there exists a big Cartier divisor $M$ on $X$ 
(see Example \ref{exAB} below). 
By the assumption, we can write $lD\sim M+E'$, 
where $E'$ is an effective $\mathbb R$-Cartier $\mathbb R$-divisor 
(see, for example, the usual proof of Kodaira's lemma:~Lemma \ref{kod-lem}). 
By using Lemma \ref{lemCDD} and Lemma \ref{lemCDDD} below, 
we can write $M+E'\equiv \sum a'_i G'_i$, where 
$a'_i$ is a positive real number and $G'_i$ is a big Cartier 
divisor for every $i$. 
By Proposition \ref{lemD}, 
$D$ is a big $\mathbb R$-divisor on $X$. 

Next, we assume (1). 
Let $f:Y\to X$ be a resolution. 
Then $f^*D$ is big by Definition \ref{defnA} and 
Lemma \ref{lem0555}. 
By Lemma \ref{lemABC} and 
Lemma \ref{lem088}, 
we obtain the desired estimate in (2). 
\end{proof}

We have already used the following lemmas in the proof of 
Proposition \ref{lemCD}. 

\begin{lem}\label{lemCDD} 
Let $X$ be a normal irreducible variety and let $B$ be an effective $\mathbb R$-Cartier 
$\mathbb R$-divisor 
on $X$. 
Then $B$ can be written in the form 
$$
B=\sum _i b_i B_i 
$$
where each $B_i$ is an effective Cartier divisor and 
$b_i$ is a positive 
real number for every $i$. 
\end{lem}
\begin{proof}
We can write $B=\sum _{j=1}^{l} d_j D_j$, 
where $d_j$ is a real number and $D_j$ is Cartier 
for every $j$. 
We set $E=\cup _j \Supp D_j$. 
Let $E=\sum _{k=1}^{m}E_k$ be the irreducible 
decomposition. 
We can write $D_j=\sum _{k=1}^{m} a^{j}_kE_k$ for every $j$. 
Note that $a^j_k$ is integer 
for every $j$ and $k$. 
We can also write $B=\sum _{k=1}^{m} c_k E_k$ with 
$c_k\geq 0$ for every $k$. 
We consider 
\begin{align*}
\mathcal E=\left\{(r_1, \cdots, r_l)\in \mathbb R^l\, \left| \, 
\sum _{j=1}^l r_j a^j_k \geq 0\ \right. \text{for every} \ k\right\}\subset \mathbb R^l. 
\end{align*}
Then $\mathcal E$ is a rational convex polyhedral cone and 
$(d_1, \cdots, d_l)\in \mathcal E$. 
Therefore, we can find effective Cartier divisors $B_i$ and positive real numbers $b_i$ such that $B=\sum _i b_i B_i$. 
\end{proof}

\begin{lem}\label{lemCDDD} 
Let $B$ be a big Cartier divisor on a normal irreducible 
variety $X$ and let $G$ be an effective 
Cartier divisor on $X$. 
Then $B+rG$ is big for any positive real number $r$. 
\end{lem}

\begin{proof}
If $r$ is rational, then this lemma is obvious by the definition of 
big $\mathbb Q$-divisors. If $r$ is not rational, then 
we can write 
$$
B+rG=t(B+r_1G)+(1-t)(B+r_2G)
$$
where $r_1$ and $r_2$ are rational, 
$0<r_1<r<r_2$, and $t$ is a real number with $0<t<1$. 
By Definition \ref{defnA}, $B+rG$ is a big $\mathbb R$-divisor.  
\end{proof}

Example \ref{exAB} implies that 
a normal complete algebraic variety does not always 
have big Cartier divisors even when the Picard number is one. 
For the details of Example \ref{exAB}, 
see \cite[Section 4]{fujino-km}.

\begin{ex}\label{exAB}
Let $\Delta$ be the fan in $\mathbb R^3$ whose 
rays are generated by $v_1=(1, 0, 1)$, 
$v_2=(0, 1, 1)$, 
$v_3=(-1, -2, 1)$, 
$v_4=(1, 0, -1)$, 
$v_5=(0, 1, -1)$, 
$v_6=(-1, -1, -1)$ and 
whose maximal cones are 
\begin{align*}
\langle v_1, v_2, v_4, v_5\rangle, 
\langle v_2, v_3, v_5, v_6\rangle, 
\langle v_1, v_3, v_4, v_6 \rangle, 
\langle v_1, v_2, v_3\rangle, 
\langle v_4, v_5, v_6 \rangle. 
\end{align*}
Then the associated toric threefold 
$X$ is complete with $\rho (X)=0$. 
More precisely, every Cartier divisor 
on $X$ is linearly equivalent to zero. 

Let $f:Y\to X$ be the blow-up 
along $v_7=(0, 0, -1)$ and let $E$ be the 
$f$-exceptional 
divisor on $Y$. 
Then we can check that $\rho (Y)=1$ and 
that $\mathcal O_Y(E)$ is a generator of 
$\Pic (Y)$. Therefore, 
there are no big Cartier divisors on $Y$. 
\end{ex}

The next lemma is almost obvious. 

\begin{lem}\label{lemC}
Let $V$ be a complete irreducible variety and let $D$ be a big 
$\mathbb R$-Cartier $\mathbb R$-divisor 
on $V$. Let $g:W\to V$ be an arbitrary proper birational morphism 
from an irreducible variety $W$. 
Then $g^*D$ is big. 
\end{lem}
\begin{proof}
By Definition \ref{defnA}, we may assume that 
$D$ is Cartier. 
We obtain the following commutative diagram. 
$$
\xymatrix{W\ar[d]_g &\ar[l]_\mu  W^\nu \ar[d]^h\\ 
V &\ar[l]^{\nu}V^\nu
}
$$ 
Here, $\mu:W^\nu\to W$ and $\nu:V^\nu\to V$ are the normalizations. 
Since $\nu^*D$ is big, $h^*\nu^*D=\mu^*g^*D$ 
is also big. 
We note that $h$ is a birational morphism 
between normal irreducible varieties (see Lemma \ref{lem0555}). 
Thus, $g^*D$ is big by Definition \ref{def-big2}. 
\end{proof}

Kodaira's lemma for big $\mathbb R$-Cartier 
$\mathbb R$-divisors on normal varieties is 
also obvious (cf.~the proof of Lemma \ref{kod-lem}). 

\begin{lem}[Kodaira's lemma for 
big $\mathbb R$-divisors on normal varieties]\label{lemE}
Let $X$ be a complete irreducible {\em{normal}} variety 
and let $D$ be a big $\mathbb R$-Cartier 
$\mathbb R$-divisor on $X$. Let $M$ be an arbitrary 
Cartier divisor on $X$. 
Then there exist a positive integer $l$ and an effective 
$\mathbb R$-Cartier $\mathbb R$-divisor 
$E$ on $X$ such that 
$lD-M\sim E$.  
\end{lem}

Finally, we discuss relatively big $\mathbb R$-divisors. 

\begin{defn}[Relatively big $\mathbb R$-divisors]\label{defnF}
Let $\pi:X\to S$ be a proper morphism 
from an irreducible 
variety $X$ onto a variety $S$ and let $D$ be an $\mathbb R$-Cartier $\mathbb R$-divisor on 
$X$. 
Then $D$ is called {\em{$\pi$-big}} (or, {\em{big over $S$}}) if $D|_{X_\eta}$ is big 
on $X_\eta$, where $X_\eta$ is the generic fiber of $\pi$. 
\end{defn}

We need the following lemma for the proof of the Kawamata--Viehweg 
vanishing theorem for $\mathbb R$-divisors. 

\begin{lem}[{cf.~\cite[Corollary 0-3-6]{kmm}}]\label{lemGH}
Let $\pi:X\to S$ be a proper surjective morphism from an irreducible variety $X$ onto 
a quasi-projective variety $S$ and let $D$ be a $\pi$-nef and $\pi$-big 
$\mathbb R$-Cartier $\mathbb R$-divisor 
on $X$. 
Then there exist a proper birational 
morphism $\mu:Y\to X$ from a smooth variety $Y$ projective over $S$ and 
divisors $F_\alpha$'s on $Y$ such that 
$\Supp \mu^*D\cup(\cup F_\alpha)$ is a simple normal crossing divisor and 
that $\mu^*D-\sum _\alpha\delta_\alpha F_\alpha$ is $\pi\circ \mu$-ample for some 
$\delta_\alpha$ with $0<\delta_\alpha\ll 1$. 
\end{lem}
We can check Lemma \ref{lemGH} by Lemma \ref{lemE} and Hironaka's 
resolution theorem. 

%%%%%%%%%%%%%%%%%%%%%%%%%%


\begin{thebibliography}{KMM}

\bibitem[Al1]{alexeev0} 
V.~Alexeev, Moduli spaces $M_{g,n}(W)$ for surfaces, 
{\em{Higher-dimensional complex varieties (Trento, 1994)}}, 1--22, de Gruyter, Berlin, 1996.

\bibitem[Al2]{alexeev} 
V.~Alexeev, Log canonical singularities and complete moduli of stable pairs, preprint (1996). 

\bibitem[Al3]{alexeev-limits} 
V.~Alexeev, Limits of stable pairs, 
Pure Appl. Math. Q. {\textbf{4}} (2008), no. 3, 
Special Issue: In honor of Fedor Bogomolov. Part 2, 767--783.

\bibitem[AH]{alexeev-hacon} 
V.~Alexeev, C.~D.~Hacon, 
Non-rational centers of log canonical singularities, 
J. Algebra {\textbf{369}} (2012), 1--15. 

\bibitem[Am]{ambro} 
F.~Ambro, 
Quasi-log varieties, 
Tr. Mat. Inst. Steklova {\textbf{240}} (2003), 
Biratsion. Geom. Linein. Sist. Konechno 
Porozhdennye Algebry, 220--239; translation in Proc. Steklov 
Inst. Math. 2003, no. 1 (240), 214--233

\bibitem[BM]{bierstone-milman} 
E.~Bierstone, P.~D.~Milman, Resolution except for minimal 
singularities I, Adv. Math. {\textbf{231}} (2012), no. 5, 3022--3053.

\bibitem[BP]{bierstone-p}
E.~Bierstone, F.~Vera Pacheco, 
Resolution of singularities of pairs preserving semi-simple normal crossings, 
Rev. R. Acad. Cienc. Exactas F\'is. Nat. Ser. A Math. RACSAM 
{\textbf{107}} (2013), no. 1, 159--188.

\bibitem[B]{birkar}
C.~Birkar, On existence of log minimal models II, 
J. Reine Angew. Math. {\textbf{658}} (2011), 99--113. 

\bibitem[F1]{fujino-app} 
O.~Fujino, Applications of Kawamata's positivity theorem, 
Proc. Japan Acad. Ser. A Math. Sci. {\textbf{75}} (1999), no. 6, 75--79. 

\bibitem[F2]{fujino-rf} 
O.~Fujino, Base point free theorem of Reid--Fukuda type, J. Math. Sci. Univ. 
Tokyo {\textbf{7}} (2000), no. 1, 1--5. 

\bibitem[F3]{fujino-abun} 
O.~Fujino, Abundance theorem for semi log canonical threefolds, 
Duke Math. J. {\textbf{102}} (2000), no. 3, 513--532. 

\bibitem[F4]{high}
O.~Fujino, 
Higher direct images of log canonical divisors, 
J. Differential Geom. {\textbf{66}} (2004), no. 3, 453--479.

\bibitem[F5]{fujino-km} 
O.~Fujino, On the Kleiman--Mori cone, 
Proc. Japan Acad. Ser. A Math. Sci. {\textbf{81}} (2005), no. 5, 80--84. 

\bibitem[F6]{fujino-what} 
O.~Fujino, 
What is log terminal?, {\em{Flips for $3$-folds and $4$-folds}}, 
49--62, Oxford Lecture Ser. Math. Appl., {\textbf{35}}, 
Oxford Univ. Press, Oxford, 2007,

\bibitem[F7]{fujino-on} 
O.~Fujino, 
On injectivity, vanishing and torsion-free theorems 
for algebraic varieties, 
Proc. Japan Acad. Ser. A Math. Sci. {\textbf{85}} (2009), no. 8, 95--100.

\bibitem[F8]{fujino-eff} 
O.~Fujino, Effective base point free theorem for log canonical pairs---Koll\'ar type theorem, 
Tohoku Math. J. (2) {\textbf{61}} (2009), no. 4, 475--481. 

\bibitem[F9]{book} 
O.~Fujino, Introduction to the log minimal model program for 
log canonical pairs, preprint (2009). 

\bibitem[F10]{fujino-qlog} 
O.~Fujino, 
Introduction to the theory of quasi-log varieties, {\em{Classification of algebraic varieties}}, 
289--303, EMS Ser. Congr. Rep., Eur. Math. Soc., Z\"urich, 2011. 

\bibitem[F11]{fujino-non} 
O.~Fujino, 
Non-vanishing theorem for log canonical pairs, 
J. Algebraic Geom. {\textbf{20}} (2011), no. 4, 771--783.

\bibitem[F12]{fujino-fund} 
O.~Fujino, 
Fundamental theorems for the log minimal model program, 
Publ. Res. Inst. Math. Sci. {\textbf{47}} (2011), no. 3, 727--789.

\bibitem[F13]{fujino-surface} 
O.~Fujino, Minimal model theory for log surfaces, 
Publ. Res. Inst. Math. Sci. {\textbf{48}} (2012), no. 2, 339--371.

\bibitem[F14]{fujino-base} 
O.~Fujino, Basepoint-free theorems:~saturation, b-divisors, and canonical 
bundle formula, Algebra Number Theory \textbf{6} (2012), 
no. 4, 797--823. 

\bibitem[F15]{fujino-vanishing} 
O.~Fujino, Vanishing theorems, 
to appear in \lq\lq Minimal models and extremal rays\rq\rq, Adv. Stud. Pure Math. 

\bibitem[F16]{fujino-sp} 
O.~Fujino, Semipositivity theorems for moduli problems, preprint (2012). 

\bibitem[F17]{fujino-inj} 
O.~Fujino, Injectivity theorems, preprint (2013). 

\bibitem[FF]{fuji-fuji} 
O.~Fujino, T.~Fujisawa, 
Variations of mixed Hodge structure and 
semi-positivity theorems, preprint (2011). 

\bibitem[FFS]{ffs} 
O.~Fujino, T.~Fujisawa, M.~Saito, 
Some remarks on the semi-positivity theorems, 
to appear in Publ. Res. Inst. Math. Sci. 

\bibitem[FG]{fujino-gongyo} 
O.~Fujino, Y.~Gongyo, Log pluricanonical representations and abundance conjecture, 
to appear in Compositio Math.  

\bibitem[FT]{fujino-tanaka} 
O.~Fujino, H.~Tanaka, On log surfaces, Proc. Japan Acad. Ser. A Math. Sci. {\textbf{88}} (2012), 
no. 8, 109--114. 

\bibitem[Ft1]{fujita} 
K.~Fujita, Simple normal crossing Fano varieties and log Fano manifolds, preprint (2012). 

\bibitem[Ft2]{fujita2} 
K.~Fujita, Semi-terminal modifications of demi-normal pairs, 
preprint (2013). 

\bibitem[G]{gongyo} 
Y.~Gongyo, Abundance theorem for numerically trivial log canonical divisors of 
semi-log canonical pairs, 
J. Algebraic Geom. {\textbf{22}} (2013), no. 3, 549--564.

\bibitem[HK]{hk} 
C.~D.~Hacon, S.~J.~Kov\'acs, {\em{Classification of higher dimensional algebraic varieties}}, 
Oberwolfach Seminars, {\textbf{41}}. Birkh\"auser Verlag, Basel, 2010. 

\bibitem[HX]{hacon-xu} 
C.~D.~Hacon, C.~Xu, 
On finiteness of $B$-representation and semi-log canonical abundance, preprint (2011). 

\bibitem[Ha]{harts} 
R.~Hartshorne, Generalized divisors on Gorenstein schemes, 
{\em{Proceedings of Conference on Algebraic Geometry and Ring Theory in honor of Michael Artin, Part III 
(Antwerp, 1992)}}, $K$-Theory {\textbf{8}} (1994), no. 3, 287--339. 

\bibitem[KMM]{kmm} 
Y.~Kawamata, K.~Matsuda, K.~Matsuki, 
Introduction to the minimal model problem, {\em{Algebraic geometry, 
Sendai, 1985}}, 283--360, Adv. Stud. Pure Math., {\textbf{10}}, North-Holland, Amsterdam, 1987. 

\bibitem[Ko1]{kollar-proj} 
J.~Koll\'ar, 
Projectivity of complete moduli, 
J. Differential Geom. {\textbf{32}} (1990), no. 1, 235--268.

\bibitem[Ko2]{ko-ex} 
J.~Koll\'ar,  Two examples of surfaces with normal crossing singularities, 
Sci. China Math. {\textbf{54}} (2011), no. 8, 1707--1712. 

\bibitem[Ko3]{kollar-local} 
J.~Koll\'ar, A local version of the Kawamata--Viehweg vanishing theorem, 
Pure Appl. Math. Q. {\textbf{7}} (2011), no. 4, Special Issue: In memory of Eckart Viehweg, 1477--1494.

\bibitem[Ko4]{kollar} 
J.~Koll\'ar, Moduli of varieties of general type, 
{\em{Handbook of Moduli:~Volume II}}, 
131--157, 
Adv. Lect. Math. (ALM), {\textbf{25}}, Int. Press, Somerville, MA, 2013.

\bibitem[Ko5]{kollar-book} 
J.~Koll\'ar, {\em{Singularities of the Minimal Model Program}}, 
With the collaboration of S.~Kov\'acs. 
Cambridge Tracts in Mathematics, {\textbf{200}}. Cambridge University Press, Cambridge, 2013. 

\bibitem[KK]{kk} 
J.~Koll\'ar, S.~J.~Kov\'acs, 
Log canonical singularities are Du Bois, 
J. Amer. Math. Soc. {\textbf{23}} (2010), no. 3, 791--813.

\bibitem[KM]{km} 
J.~Koll\'ar, S.~Mori, {\em{Birational geometry of algebraic varieties}},  
With the collaboration of C.~H.~Clemens and A.~Corti, 
Translated from the 1998 Japanese original. Cambridge Tracts in Mathematics, {\textbf{134}}. Cambridge 
University Press, Cambridge, 1998. 

\bibitem[KSB]{ks} 
J.~Koll\'ar, N.~I.~Shepherd-Barron, 
Threefolds and deformations of surface singularities, 
Invent. Math. {\textbf{91}} (1988), no. 2, 299--338. 

\bibitem[K+]{fa} 
J.~Koll\'ar, et al., {\em{Flips and Abundance for algebraic threefolds}}, 
Papers from the Second Summer Seminar on Algebraic Geometry held at the University of Utah, 
Salt Lake City, Utah, August 1991. Ast\'erisque No. {\textbf{211}} (1992). Soci\'et\'e Math\'ematique 
de France, Paris, 1992.

\bibitem[Kv1]{kovacs} 
S.~J.~Kov\'acs, 
Young person's guide to moduli of higher dimensional varieties, 
{\em{Algebraic geometry---Seattle 2005. Part 2}}, 711--743, 
Proc. Sympos. Pure Math., {\textbf{80}}, Part 2, Amer. Math. Soc., Providence, RI, 2009. 

\bibitem[Kv2]{kovacs-irr} 
S.~J.~Kov\'acs, 
Irrational centers, 
Pure Appl. Math. Q. {\textbf{7}} (2011), 
no. 4, Special Issue: In memory of Eckart Viehweg, 1495--1515.

\bibitem[Kv3]{kovacs2} 
S.~J.~Kov\'acs, Singularities of stable varieties, {\em{Handbook of Moduli:~Volume II}}, 
159--203, 
Adv. Lect. Math. (ALM), {\textbf{25}}, Int. Press, Somerville, MA, 2013.

\bibitem[KSS]{kss} 
S.~J.~Kov\'acs, K.~Schwede, K.~E.~Smith, 
The canonical sheaf of Du Bois singularities, 
Adv. Math. {\textbf{224}} (2010), no. 4, 1618--1640.

\bibitem[L]{lposi}
R.~Lazarsfeld, {\em{Positivity in algebraic geometry. I. Classical setting: line bundles and linear series}}, 
Ergebnisse der Mathematik und ihrer Grenzgebiete. 3. Folge. A Series of Modern Surveys in Mathematics, 
{\textbf{48}}. Springer-Verlag, Berlin, 2004. 

\bibitem[N]{nakayama2}
N.~Nakayama, {\em{Zariski-decomposition and abundance}},  
MSJ Memoirs, {\textbf{14}}. Mathematical Society of Japan, Tokyo, 2004. 

\bibitem[O]{odaka} 
Y.~Odaka, The GIT stability of polarized varieties via discrepancy, 
Ann. of Math. (2) {\textbf{177}} (2013), no. 2, 645--661.

\bibitem[T1]{tanaka} 
H.~Tanaka, Minimal models and abundance for positive characteristic log surfaces, 
to appear in Nagoya Math. J. 

\bibitem[T2]{tanaka2} 
H.~Tanaka, X-method for klt surfaces in positive characteristic, 
to appear in J. Algebraic Geom. 

\end{thebibliography}
\end{document}